\documentclass[12pt]{article}

\input{style}

\usepackage{amsmath, amssymb, amstext, amsfonts, mathrsfs}

\begin{document}

\title{\vspace*{-3cm} Linear quadratic stochastic control problems
  with stochastic terminal constraint\footnote{We are grateful to
    Paulwin Graewe, Ulrich Horst, Peter Imkeller and Alexandre Popier
    for insightful discussions. Our paper has also greatly benefited
    from the comments and suggestions of two anonymous referees.}}

\author{ Peter Bank\footnote{Technische Universit{\"a}t Berlin,
    Institut f{\"u}r Mathematik, Stra{\ss}e des 17. Juni 136, 10623
    Berlin, Germany, email \texttt{bank@math.tu-berlin.de}.  Financial
    support by Einstein Foundation through project ``Game options and
    markets with frictions'' is gratefully acknowledged.}
  \hspace{4ex} Moritz Vo{\ss}\footnote{Technische Universit{\"a}t
    Berlin, Institut f{\"u}r Mathematik, Stra{\ss}e des 17. Juni 136,
    10623 Berlin, Germany, email \texttt{voss@math.tu-berlin.de}.}  }
\date{\today}

\maketitle
\begin{abstract}
  We study a linear quadratic optimal control problem with stochastic
  coefficients and a terminal state constraint, which may be in force
  merely on a set with positive, but not necessarily full
  probability. Under such a partial terminal constraint, the usual
  approach via a coupled system of a backward stochastic Riccati
  equation and a linear backward equation breaks down. As a remedy, we
  introduce a family of auxiliary problems parametrized by the
  supersolutions to this Riccati equation alone. The target functional
  of these problems dominates the original constrained one and allows
  for an explicit description of both the optimal control policy and
  the auxiliary problem's value in terms of a suitably constructed
  optimal signal process. This suggests that, for the minimal
  supersolution of the Riccati equation, the minimizers of the
  auxiliary problem coincide with those of the original problem, a
  conjecture that we see confirmed in all examples understood so far.
\end{abstract}

\begin{description}
\item[Mathematical Subject Classification (2010):] 93E20, 60H10,
  91G80
\item[Keywords:] linear quadratic control, stochastic terminal
  constraint, backward stochastic Riccati differential equation,
  optimal tracking, optimal signal process
\end{description}


\section{Introduction}
Linear quadratic stochastic optimal control problems (stochastic LQ
problems in short) represent an important class of stochastic control
problems and are very well studied in the literature, cf., e.g., the
book by \citet{YongZhou:99}, Chapter 6, for an overview. A prototype
of a stochastic LQ problem with linear quadratic cost functional is
given by the so-called \emph{optimal follower} or \emph{optimal
  tracking} problem where one seeks to minimize a cost criterion of
the following form: For a deterministic time horizon $T>0$, for a
predictable target process $(\xi_t)_{0 \leq t \leq T}$ as well as
progressively measurable, nonnegative processes
$(\nu_t)_{0 \leq t \leq T}$ and $(\kappa_t)_{0 \leq t \leq T}$, for
random variables $\eta$ and $\Xi_T$ known at time $T$ and $x \in \RR$,
find a control $u$ with state process
\begin{equation*}\label{eq:stateprocess}
X^u_t = x + \int_0^t u_s ds \quad (0 \leq t \leq T)
\end{equation*}
which minimizes the objective
\begin{equation} \label{eq:objectiveintro}
  J^\eta(u) \set \EE \left[ \int_0^T (X^u_t - \xi_t)^2 \nu_t dt 
   + \int_0^T \kappa_t u^2_t dt + \eta
   (X^u_T - \Xi_T)^2 \right].
\end{equation}
The interpretation of such an LQ problem is the following: The first
term in~\eqref{eq:objectiveintro} measures the overall quadratic
deviation of the controlled state process $X^u$ from the target
process $\xi$ weighted with a stochastic weight process $\nu$. The
second term in \eqref{eq:objectiveintro} measures the incurred
tracking effort in terms of running quadratic costs which are imposed
on the control~$u$ with stochastic cost process $\kappa$. The third
term in \eqref{eq:objectiveintro} implements a penalization on the
quadratic deviation of the controlled state $X^u_T$ from the final
target position~$\Xi_T$ at terminal time $T$ with nonnegative random
penalization parameter~$\eta$.

 It is well known in the literature that
the optimal control to such a stochastic LQ problem as well as its
optimal value is typically characterized by two coupled backward
stochastic differential equations (BSDEs): A backward stochastic
\emph{Riccati} differential equation (BSRDE) of the form
\begin{equation} \label{eq:BSRDEintro} 
  dc_t = \left(
    \frac{c_t^2}{\kappa_t} - \nu_t \right) dt - dN_t \quad \text{on }
  [0,T] \text{ with } c_T = \eta
\end{equation}
and a linear BSDE of the form
\begin{equation} \label{eq:linBSDEintro} 
  db_t = \left(
    \frac{c_t}{\kappa_t} b_t - \nu_t \xi_t \right) dt + dM_t \quad \text{on }
  [0,T] \text{ with } b_T = \eta \Xi_T,
\end{equation}
where $N$ and $M$ denote suitable c\`adl\`ag martingales (cf., e.g.,
\citet{KohlmannTang:02}, Section 5.1).

A number of interesting challenges arise when one allows the terminal
penalization parameter $\eta$ to take the value infinity with positive
(not necessarily full) probability. It is then intuitively sensible to
interpret the ``blow up'' of $\eta$ as a \emph{stochastic
terminal state} constraint of the form
\begin{equation} \label{eq:tcintro} 
  X^u_T = \Xi_T \quad \text{a.e. on the set } \{ \eta = + \infty \}
\end{equation}
on all controlled processes $X^u$ that produce a finite value in
\eqref{eq:objectiveintro}. Mathematically, it is less obvious how to
tackle this delicate ``partial'' constraint and how to compute the
optimal control as well as the optimal value. Indeed, note that the
involved BS(R)DEs in \eqref{eq:BSRDEintro} and \eqref{eq:linBSDEintro}
will both now exhibit with positive probability a \emph{singularity at final
time} in this case. The possibly singular BSRDE in
\eqref{eq:BSRDEintro} does not pose a serious problem; see
\citet{KrusePopier:16_1} and \citet{Popier:16}. In contrast, the
singularity in the terminal condition of the linear BSDE in
\eqref{eq:linBSDEintro} is rather unpleasant because it also involves
the desired target position $\Xi_T$, leaving the terminal condition
$b_T = \eta \Xi_T$ depend solely on the sign of $\Xi_T$ on the very set
$\{\eta=+\infty\}$ where this random variable has to be matched by the
state processes' terminal value $X^u_T$. 

As a consequence, the classical solution approach cannot be followed
directly. Instead we introduce a family of auxiliary target
functionals
$$
J^c(u) \set \limsup_{\tau \uparrow T} \EE \left[ \int_0^\tau (X^u_t - \xi_t)^2 \nu_t dt +
    \int_0^\tau \kappa_t u^2_t dt + c_\tau (X^u_\tau -
    \hat{\xi}^c_\tau)^2 \right]
$$ 
parametrized by supersolutions $c$ of the BSRDE~\eqref{eq:BSRDEintro}
and where $\hat{\xi}^c_\tau$ is an \emph{optimal signal process} constructed
as a judiciously chosen average of future target positions
$(\xi_t)_{t \geq \tau}$ and $\Xi_T$. The target functional $J^c$ avoids
the singularity at time $T$ by a ``truncation in time'' focussing
on shorter time horizons $\tau<T$ at which we impose a ``classical''
finite terminal penalization. This penalization is chosen in such a
way that the corresponding optimizers can be extended consistently to
the full interval $[0,T)$ as $\tau \uparrow T$. In fact, the
corresponding auxiliary minimization problems turn out to be solvable
in a very satisfactory way: As already observed in a much simpler
setting in~\citet{BankSonerVoss:16}, we can give necessary and
sufficient conditions for the domain $\{J^c<\infty\}$ to be nonempty
and we can also describe explicitly the optimal control in feedback
form
$$
\hat{u}^c_t = \frac{c_t}{\kappa_t}(\hat{\xi}^c_t-X^{\hat{u}^c}_t)
\quad (0 \leq t < T),
$$
revealing that one should always push the controlled process towards
the optimal signal $\hat{\xi}^c$ with time-varying urgency given by
the ratio $c_t/\kappa_t$. We can even show how the regularity and
predictability of the targets $\xi$ and $\Xi_T$ as reflected in the
signal process $\hat{\xi}^c$ and its quadratic variation determine the
problem's value.

We show that for the considered supersolutions $c$ of the
BSRDE~\eqref{eq:BSRDEintro} we have $J^c(u) \geq J^\eta(u)$. This
leads us to the conjecture that for the \emph{minimal} supersolution
$c=c^{\min}$ (whose existence is guaranteed under mild conditions; see
\citet{KrusePopier:16_1}) the minimizers of these functionals are the same. While we
have to leave the proof of this conjecture for future research that
allows one to better control singular BSRDE supersolutions, we do verify
the validity of our conjecture in the examples we found in the literature.

Stochastic control problems, referred to as optimal liquidation
problems in the literature, with almost sure (i.e., $\eta \equiv + \infty$
almost surely) and deterministic terminal state constraint (targeting
the terminal position $\Xi_T = 0$), where the cost functional is
allowed to be quadratic in $X^u$ and $u$ (that is, $\xi \equiv 0$ in
\eqref{eq:objectiveintro}) have already been studied in, e.g.,
\citet{Schied:13}, \citet{AnkirchnerJeanblancKruse:14} and, in a more
general BSPDE framework, in \citet{GraeweHorstQiu:15}; allowing the
penalization parameter~$\eta$ to take the value infinity with positive
probability has been investigated in
\citet{KrusePopier:16_1}. \citet{AnkirchnerKruse:15}, still within
this context of optimal liquidation, allow the objective functional to
be additionally linear in the control $u$. They also incorporate a
specific nonzero stochastic terminal state constraint where the random
target position $\Xi_T$ is gradually revealed up to terminal time
$T$. A general class of stochastic control problems including LQ
problems with terminal states being constrained to a convex set were
studied by \citet{JiZhou:06}. However, to the best of our knowledge,
stochastic linear quadratic control problems with $\xi \neq 0$ and
possible stochastic terminal state constraint $\Xi_T \neq 0$ as
considered in the present paper have not yet been investigated.

The analysis of the stochastic LQ problem in \eqref{eq:objectiveintro}
above is especially motivated by optimal trading and hedging problems
in Mathematical Finance. In this framework the state process $X^u$
denotes an agent's position in some risky asset that she trades at a
turnover rate $u$. She wants her position to be as close as possible
to a given target strategy $\xi$ but simultaneously seeks to minimize
the induced quadratic transaction costs which are levied on her
transactions due to, e.g., stochastic price impact as measured by
$\kappa$. The weight process $\nu$ captures stochastic volatility,
that is, the risk of her open trading position due to random market
fluctuations. Finally, in case of a possible but not necessarily
almost sure occurrence of specific market conditions, encoded by the
event set $\{ \eta = + \infty \}$, she may be required to drive her
position~$X^u$ imperatively towards a predetermined random value
$\Xi_T$ at maturity~$T$ (e.g., to respect specific requirements of
contractual or regulatory nature concerning her risky asset
position). Otherwise, a penalization depending on the deviation of
$X^u_T$ from the target position $\Xi_T$ is implemented. We refer to,
e.g., \citet{RogerSin:10}, \citet{NaujWes:11}, \citet{AlmgrenLi:16},
\citet{FreiWes:13}, \citet{CarJai:15}, \citet{CaiRosenbaumTankov:15},
\citet{BankSonerVoss:16} and, for asymptotic considerations, to
\citet{ChanSircar:16}. Note, however, that the above cited papers may
neither allow for an arbitrary predictable target strategy $\xi$ nor
for stochastic price impact $\kappa$ and stochastic
volatility~$\nu$. In particular, none of them consider a stochastic
terminal state constraint like~\eqref{eq:tcintro} above with general
random target position~$\Xi_T$.

The rest of the paper is organized as follows. In Section
\ref{sec:stochLQproblem} we formulate the general stochastic LQ
problem with stochastic terminal state constraint. Our auxiliary
control problem and its solution are presented in
Section~\ref{sec:auxproblem}. Its relation to the original LQ problem
is discussed and exemplarily illustrated in
Section~\ref{sec:illustration}. The proofs are deferred to
Section~\ref{sec:proofs} and an appendix collects a
few BSDE-results which may be of independent interest.


\section{A stochastic LQ problem with stochastic terminal state
  constraint}
\label{sec:stochLQproblem}

We fix a finite deterministic time horizon $T > 0$ and a filtered
probability space $(\Omega,\mathcal{F},(\cF_t)_{0 \leq t \leq T},\PP)$
satisfying the usual conditions of right continuity and completeness.
We let $(\kappa_t)_{0 \leq t \leq T}$ and $(\nu_t)_{0 \leq t \leq T}$
denote two progressively measurable, nonnegative processes such
that
\begin{equation} \label{eq:condkappanu}
  \int_0^T \left(\nu_t + \frac{1}{\kappa_t} \right) dt <
  \infty \quad \PP\text{-a.s.}
\end{equation}
Moreover, we are given a predictable target process
$(\xi_t)_{0 \leq t \leq T}$ satisfying
\begin{equation} \label{eq:condxi} 
  \EE \left[ \int_0^T \vert \xi_t \vert \nu_t dt \right]
    < \infty \quad \text{and} \quad \int_0^T \xi^2_t \nu_t dt < \infty
    \quad \PP\text{-a.s.},
\end{equation}
a random terminal target position $\Xi_T \in L^0(\PP,\cF_{T-})$ as
well as an~$\cF_{T-}$-measurable penalization parameter
$\eta$ taking values in $[0,+\infty]$. We further assume
that
\begin{equation} \label{ass:etakappa}
  \PP\left[ \eta = 0\, , \int_t^T \nu_u du = 0 \, \bigg\vert\,
  \cF_t\right] < 1 \quad \PP\text{-a.s. for all } t \in [0,T).
\end{equation}
For such $\nu, \kappa, \xi, \Xi_T$ and $\eta$,
one can formulate the following stochastic linear quadratic
optimal control problem: Find a control $u$ from the class of
processes
\begin{equation} \label{eq:setU}
\cU \set \left\{ u \text{ progressively
    measurable s.t.} \int_0^T \vert u_t \vert \, dt < \infty \text{ $\PP$-a.s.} \right\}
\end{equation}
such that, for given $x \in \mathbb{R}$, the controlled state process
\begin{equation} \label{eq:defX}
  X^u_t \set x + \int_0^t u_s ds \quad (0 \leq t \leq T)
\end{equation}
minimizes the objective functional
\begin{equation} \label{eq:originalobjective}
   J^\eta(u) \set \EE \left[ \int_0^T (X^u_t - \xi_t)^2 \nu_t dt 
   + \int_0^T \kappa_t u^2_t dt + \eta 1_{\{ 0 \leq \eta < \infty \}}
   (X^u_T - \Xi_T)^2 \right]
\end{equation}
over the set of all constrained policies
\begin{align} \label{def:originalpolicies}
\begin{aligned}
\cU^{\Xi} \set \Big\{ u \in \cU \text{ satisfying }  X^u_T = \Xi_T
\text{ $\PP$-a.s. on } \{ \eta = + \infty\}  \Big\}.
\end{aligned}
\end{align}
In short, we are interested in the stochastic LQ problem
\begin{equation} \label{eq:originalLQP}
J^\eta(u) \rightarrow \min_{u \in \cU^{\Xi}},
\end{equation}
where the controller seeks to keep the controlled process $X^u$ close
to a given target process $\xi$ in such a way that deviations from
the final target position~$\Xi_T$ are also minimized. On
$\{ \eta = + \infty \}$, the final target position has to be reached a.s. as
incorporated in the set of admissible strategies $\cU^{\Xi}$
in~\eqref{def:originalpolicies}.

\begin{Remark} \label{rem:intro} 
$\phantom{}$
\vspace{-.5em}
\begin{enumerate}
\item In case where the random penalization parameter $\eta$ is finite
  almost surely, the optimization problem in~\eqref{eq:originalLQP}
  does not include a terminal state constraint and boils down to a
  classical stochastic optimal control problem which is well studied
  in the literature; c.f., e.g.,~\citet{KohlmannTang:02}. 
\item The dynamic condition~\eqref{ass:etakappa} is
  very natural for our optimal tracking problem
  in~\eqref{eq:originalLQP}. It means that at any time $t < T$ some
  penalization for deviating from the targets $\xi$ and $\Xi_T$ remains
  conceivable, even conditionally on $\cF_t$, so that the controller
  has to stay alert all the way until the end.
\item The mild integrability conditions in \eqref{eq:condkappanu},
  \eqref{eq:condxi} and \eqref{eq:setU} ensure that the stochastic LQ
  problem in~\eqref{eq:originalLQP} is well defined along with some
  processes to be introduced shortly.
\end{enumerate}
\end{Remark}

Mathematically, the stochastic terminal state constraint
\begin{equation} \label{eq:terminalconstraint}
X^u_T = \Xi_T \quad \text{a.e. on } \{ \eta = + \infty \}
\end{equation}
in the set of allowed controls $\cU^\Xi$
in~\eqref{def:originalpolicies} entails technical difficulties. For
instance, it is far from obvious under what conditions we have
$\cU^{\Xi} \neq \varnothing$ or whether $J^\eta (u) < \infty$ for some
$u \in \cU^\Xi$. Also, as explained in the introduction, the usual
solution approach via BSDEs does not accommodate this partial
constraint.

As a possible remedy, instead of tackling the constrained stochastic
LQ problem posed in~\eqref{eq:originalLQP}, we propose to formulate a
suitable variant of this problem. Specifically, we will introduce a
family of stochastic control problems
\begin{equation} \label{eq:auxLQP}
J^c(u) \rightarrow \min_{u \in \cU^c}
\end{equation}
with set of admissible controls
\begin{equation} \label{eq:policies}
\cU^c \set \{ u \in \cU \text{ satisfying } J^c(u) < \infty \}
\end{equation}
and target functional $J^c$ which are parametrized by
supersolutions $c\set(c_t)_{0 \leq t < T}$ to a certain singular
backward stochastic differential equation (BSDE) to be described below
in Section~\ref{subsec:BSRDE}. These auxiliary problems will dominate
the stochastic LQ problem stated in~\eqref{eq:originalLQP} in the
sense that, for all parametrizations~$c$, we have
\begin{equation} \label{eq:upperboundJ}
J^c(u) \geq J^\eta(u) \quad \text{for all } u \in \cU^c
\end{equation} 
and
\begin{equation} \label{eq:subsetJ}
\cU^c \subseteq \cU^{\Xi}
\end{equation}
(cf. Lemma~\ref{lem:dominate} below). We will show in
Section~\ref{subsec:mainresult} that our auxiliary problems
in~\eqref{eq:auxLQP} can be solved in a very satisfactory way: In
terms of $\xi$, $\Xi_T$ and the parameter process $c$, we provide
necessary and sufficient conditions which ensure that
$\cU^c \neq \varnothing$ and describe explicitly the optimal 
policy $\hat{u}^c$ for~\eqref{eq:auxLQP} as well as the associated
minimal costs $J^c(\hat{u}^c)$. In view of~\eqref{eq:upperboundJ}
and~\eqref{eq:subsetJ}, we thus obtain both explicit candidate
strategies for the general constrained stochastic LQ problem
formulated in~\eqref{eq:originalLQP} as well as conditions which
guarantee existence of controls entailing finite costs in the
latter.

To link these problems to the original problem~\eqref{eq:originalLQP}
it is natural to consider ``small'' solutions to the BSDE. In fact, we
conjecture that for the \emph{minimal supersolution} $c^{\min}$ of the
BSDE we have
\begin{equation} \label{eq:conjecture}
\argmin_{\cU^\Xi} J^\eta = \argmin_{\cU^{c^{\min}}} J^{c^{\min}}.
\end{equation}
While we cannot prove this conjecture in full generality, we provide
in Section~\ref{sec:illustration} evidence for its validity in certain
settings. These include the case of bounded coefficients, but also
some singular cases where, possibly, $\PP[\eta = +\infty] > 0$.


\section{An auxiliary control problem}
\label{sec:auxproblem}

In this section, we will formulate and solve our auxiliary stochastic
LQ problem~\eqref{eq:auxLQP} for fixed $c$. The process $c$ will be a
supersolution to a BSRDE which we discuss in
Section~\ref{subsec:BSRDE}. In
Section~\ref{subsec:problemformulation}, we will introduce our target
functional $J^c$ whose minimizer $\hat{u}^c$ is derived in
Section~\ref{subsec:mainresult} along with the optimal costs
$J^c(\hat{u}^c)$.

\subsection{Connection between stochastic LQ problems and BSRDEs}
\label{subsec:BSRDE}

It is well known in the literature that the solution to stochastic
LQ problems like~\eqref{eq:originalLQP} is intimately related to backward
stochastic Riccati differential equations (BSRDEs):
For~\eqref{eq:originalLQP}, the Riccati dynamics take the form
\begin{equation} \label{eq:BSRDE} 
  dc_t = \left( \frac{c_t^2}{\kappa_t}
    - \nu_t \right) dt - dN_t \quad \text{on } [0,T)
\end{equation}
for some c\`adl\`ag martingale $(N_t)_{0 \leq t < T}$; cf.,
e.g.,~\citet{Bismut:76,Bismut:78}. Moreover, the recent papers by,
e.g.,~\citet{AnkirchnerJeanblancKruse:14},~\citet{KrusePopier:16_1},~\citet{GraeweHorstQiu:15}
or~\citet{GraeweHorst:17} have shown that a terminal state constraint
as~\eqref{eq:terminalconstraint} in the LQ problem typically leads to
a singular terminal condition for the corresponding BSRDE of the form
\begin{equation} \label{eq:BSRDEtc} 
  \liminf_{t \uparrow T} c_t \geq
  \eta \quad \PP\text{-a.s.}
\end{equation}
This motivates us to let $c = (c_t)_{0 \leq t < T}$ denote from now on
an $(\cF_t)_{0 \leq t < T}$-adapted, c\`adl\`ag semimartingale with
BSRDE dynamics~\eqref{eq:BSRDE} and terminal
condition~\eqref{eq:BSRDEtc}. In addition, we will assume that
\begin{equation} \label{eq:BSRDEintcond2} 
  \int_{[0,T)}
  \frac{d[c]_t}{c_{t-}^2} < \infty \quad \text{on the set } \{ \eta =
  + \infty\},
\end{equation}
where $[c]$ denotes the quadratic variation process of~$c$ (cf., e.g.,
\citet{Prot:04}, Chapter II.6, for the quadratic variation process of
c\`adl\`ag semimartingales).

\begin{Remark} \label{rem:BSRDE} 
  \begin{enumerate}
  \item As usual the dynamics in \eqref{eq:BSRDE} have to be
    understood in the sense that the pair $(c,N)$ satisfies
    \begin{equation} \label{eq:BSRDEintegral}
      c_s = c_t - \int_s^t \left(
        \frac{c_u^2}{\kappa_u} - \nu_u \right) du + \int_s^t dN_u
      \quad (0 \leq s \leq t <T).
    \end{equation}
    In particular, the dynamics in \eqref{eq:BSRDE} are only required
    to hold on $[0,T-\epsilon]$ for every $\epsilon > 0$, that is,
    strictly before~$T$. So, more precisely, we will say that $(c,N)$
    is a supersolution of the BSRDE~\eqref{eq:BSRDE} with terminal
    condition $\eta$ if~\eqref{eq:BSRDEintegral}
    and~\eqref{eq:BSRDEtc} hold true.

  \item For bounded coefficients $\nu$, $\kappa$, $1/\kappa$,
    $\eta$,~\citet{KohlmannTang:02} prove within a Brownian framework
    existence and uniqueness of $c$ with dynamics in~\eqref{eq:BSRDE}
    such that $\lim_{t \uparrow T} c_t = \eta$ exists $\PP$-a.s. For
    the fully singular case $\eta \equiv + \infty$ $\PP$-a.s. and
    again within a Brownian framework, existence of a minimal solution
    (under suitable integrability conditions on the processes
    $(\nu_t)_{0 \leq t \leq T}$ and $(\kappa_t)_{0 \leq t \leq T}$) to
    the above BSRDE in~\eqref{eq:BSRDE} with singular terminal
    condition $\liminf_{t \uparrow T} c_t = + \infty$ $\PP$-a.s. are
    provided in \citet{AnkirchnerJeanblancKruse:14}; cf. also
    \citet{GraeweHorstQiu:15}. For the present partially singular setup,
    \citet{KrusePopier:16_1} provide sufficient conditions (including
    suitable integrability conditions on
    $(\kappa_t)_{0 \leq t \leq T}$ and $(\nu_t)_{0 \leq t \leq T}$)
    for the existence of a minimal supersolution
    $(c^{\min}_t)_{0 \leq t \leq T}$ to the above BSRDE
    in~\eqref{eq:BSRDE} with terminal condition~\eqref{eq:BSRDEtc} in
    the sense that $c_t^{\min} \leq c_t$ for all $t \in [0,T)$ and all
    processes $c$ satisfying likewise~\eqref{eq:BSRDE}
    and~\eqref{eq:BSRDEtc}. Existence of actual solutions $c$ with
    $\lim_{t \uparrow T} c_t=\eta$ is only known under additional
    assumptions on $\eta$; see~\citet{Popier:16}.

  \item The additional integrability
    condition~\eqref{eq:BSRDEintcond2} on the ``blow up'' set
    $\{ \eta = + \infty \}$ is implicitly shown to hold true
    in~\citet{Popier:06} in a Brownian framework for constant
    coefficients $\nu \equiv 0$ and $\kappa \equiv 1$; see Theorem 2
    and Proposition 3 in \cite{Popier:06}. We require this
    integrability condition~\eqref{eq:BSRDEintcond2} in our proof of
    Lemma~\ref{lem:L} below whose result crucially feeds into our
    solution presented in Section~\ref{subsec:mainresult}. We will
    therefore briefly discuss exemplarily in the appendix sufficient
    conditions on $(\kappa_t)_{0 \leq t \leq T}$,
    $(\nu_t)_{0 \leq t \leq T}$ and $\eta$ under which property
    \eqref{eq:BSRDEintcond2} does hold true in the more generic
    setting of~\citet{KrusePopier:16_1}.
  \end{enumerate}
\end{Remark}

As a consequence of~\eqref{eq:BSRDE}, \eqref{eq:BSRDEtc}
and~\eqref{eq:BSRDEintcond2}, let us first ascertain that $c$ is
strictly positive on $[0,T)$, a result which is crucial for our
approach below and which follows immediately from Lemma~\ref{app:lem:lbound}
in the appendix. 

\begin{Lemma} \label{lem:cpositivity} For all $t \in [0,T)$ we have
  $c_t > 0$ if \eqref{ass:etakappa} holds true. \qed
\end{Lemma}

Next, the BSRDE supersolution $c$ gives rise to the following auxiliary
process
\begin{equation} \label{eq:defL} 
  L_t \set L^c_t \set c_t \exp \left(-\int_0^t
    \frac{c_u}{\kappa_u} du \right) \quad (0 \leq t < T).
\end{equation}

\begin{Lemma} \label{lem:L} Granted~\eqref{ass:etakappa} holds true,
  the process $(L_t)_{0 \leq t < T}$ is a strictly positive c\`adl\`ag
  supermartingale. In particular,
  \begin{equation} \label{eq:defLlimit} 
    L_T \set\lim_{t
      \uparrow T} L_t \geq 0 \quad \text{exists } \PP\text{-a.s.}
  \end{equation}
  and the extended process $(L_t)_{0 \leq t \leq T}$ is a supermartingale on
  $[0,T]$. Moreover, we have $\{ \eta > 0\} \subset \{ L_T > 0\}$ up
  to a $\PP$-null set.
\end{Lemma}

\begin{proof}
  Since $c_t > 0$ $\PP$-a.s. for all $0 \leq t < T$ by
  Lemma~\ref{lem:cpositivity}, it is immediate from \eqref{eq:defL}
  that also $L_t > 0$ $\PP$-a.s. for all $0 \leq t <T$. Integration by
  parts and using the Riccati dynamics of $c$ in \eqref{eq:BSRDE}
  yields that~$L$ satisfies the stochastic differential equation
  \begin{equation}
    L_0 = c_0, \quad dL_t = 
    L_{t-} \left( - \frac{\nu_t}{c_{t-}} dt - \frac{1}{c_{t-}}
      dN_{t} \right) \quad \text{on } [0,T). \label{eq:SDEL}
  \end{equation}
  Since $N$ is a c\`adl\`ag local martingale on $[0,T)$, we obtain
  from~\eqref{eq:SDEL} that the process $L$ is a strictly positive
  c\`adl\`ag supermartingale on $[0,T)$. Hence, it follows by the
  (super-)martingale convergence theorem (see, e.g.,
  \citet{KaratShr:91}, Chapter 1.3, Problem 3.16) that the limit
  $L_T \set \lim_{t \uparrow T} L_t$ exists $\PP$-a.s. and extends the
  process $L$ to a c\`adl\`ag supermartingale on all of
  $[0,T]$. Moreover, appealing to the definiton of $L$ in
  \eqref{eq:defL} and the terminal condition
  $\liminf_{t \uparrow T} c_t \geq \eta$ of the process $c$
  in~\eqref{eq:BSRDEtc}, we have
  $\{0 < \eta < \infty \} \subset \{ L_T > 0\}$. Concerning the ``blow
  up'' set $\{ \eta = + \infty \}$, observe that we may write
  \begin{equation} \label{eq:solL}
    L_t = c_0 e^{X_t - \frac{1}{2}[X]^{\mathbf{c}}_t} \prod_{s \leq t} (1 + \Delta
    X_s) e^{-\Delta X_s} \quad (0 \leq t < T),
  \end{equation}
  where
  $X_t \set - \int_0^t \frac{\nu_s}{c_{s-}} ds - \int_0^t
  \frac{1}{c_{s-}} dN_{s}$ and where $[X]^{\mathbf{c}}$ denotes the
  continuous part of its quadratic variation (cf., e.g.,
  \citet{Prot:04}, Theorem II.37). Note that $L_s > 0$ $\PP$-a.s. for
  all $0 \leq s < T$ implies $\Delta X_s > -1$ for all $0 \leq s <
  T$. Moreover, applying Taylor's formula, it holds for all
  $0 \leq t < T$ that
  \begin{equation*}
    \sum_{s \leq t} \left\vert \log \left( (1 + \Delta
    X_s) e^{-\Delta X_s} \right) \right\vert \leq \frac{1}{2} \int_{[0,T)}
    \frac{1}{c^2_{s-}}d[c]_s < +\infty
  \end{equation*}
  a.e. on the set $\{ \eta = + \infty \}$ by virtue of condition
  \eqref{eq:BSRDEintcond2}. This implies that the product of the jumps
  in \eqref{eq:solL} will converge to a strictly positive limit as
  $t \uparrow T$ on $\{ \eta = + \infty \}$. Concerning the limiting
  behaviour of the exponential $\exp(X_t - \frac{1}{2}[X]^{\mathbf{c}}_t)$ in
  \eqref{eq:solL} for $t \uparrow T$, observe that once more condition
  \eqref{eq:BSRDEintcond2} prevents the limiting value from becoming 0
  on $\{ \eta = + \infty \}$. Indeed, the local martingale
  $\int_0^{t} dN_{s}/c_{s-}$ cannot explode as
  $t \uparrow T$ for those paths along which its quadratic variation
  $\int_0^t d[c]_s/c^2_{s-}$ remains bounded on $[0,T)$ (cf., e.g.,
  \citet{Prot:04}, Chapter V.2, for more details).
\end{proof}


\subsection{Auxiliary target functional}
\label{subsec:problemformulation}

Let us assume that the terminal target position $\Xi_T$ is bounded,
or, more generally, that it satisfies
\begin{equation} \label{ass:Xi}
  \Xi_T L_T \in L^1(\cF_{T-},\PP), 
\end{equation}
where
$L_T = L_T^c=\lim_{t \uparrow T} c_t e^{-\int_0^t c_u/\kappa_u du}$ as
in~\eqref{eq:defLlimit}. Recalling the integrability
requirement~\eqref{eq:condxi} for the running target $\xi$, let us now
introduce the key object for our approach, the optimal signal process
$\hat{\xi}$ which is given by the c\`adl\`ag semimartingale
\begin{equation} \label{eq:defoptsignal} 
  \hat{\xi}_t \set \hat{\xi}^c_t \set
  \frac{1}{L_t} \EE\left[ \Xi_T L_T + \int_t^T \xi_r e^{-\int_0^r
      \frac{c_u}{\kappa_u} du} \nu_r dr \, \bigg\vert \, \cF_t \right]
  \quad (0 \leq t < T).
\end{equation}
The optimal signal process $\hat{\xi}$ can be viewed as a weighted
average of expected future targets $\xi$ and $\Xi_T$; see our
discussion in Remark~\ref{rem:osp} and the representation of
$\hat{\xi}$ in \eqref{eq:repxi} below. Our motivation for introducing
$\hat{\xi}$ becomes very apparent when reviewing known results in the
literature to the stochastic LQ problem in~\eqref{eq:originalLQP} with
bounded coefficients; see Section 4.1 below. Observe that $\hat{\xi}$
remains unspecified for $t=T$. In fact, we can readily deduce from
Lemma~\ref{lem:L} and the integrability conditions~\eqref{eq:condxi}
and~\eqref{ass:Xi}
\begin{equation} \label{eq:limitos} 
  \exists \, \lim_{t \uparrow T}
  \hat{\xi}_t = \Xi_T \quad \text{on the set } \{ L_T > 0 \} \supset
  \{ \eta > 0 \}.
\end{equation}
On the set $\{ L_T = 0 \} \subset \{\eta = 0\}$, though, this
convergence may fail (without harm as it turns out).

Given the optimal signal process $(\hat{\xi}_t)_{0 \leq t < T}$, we
are now in a position to introduce the auxiliary LQ target functional
\begin{equation} \label{eq:defobjective} 
  J^c(u) \set \limsup_{\tau \,
    \uparrow \, T} \EE \left[ \int_0^\tau (X^u_t - \xi_t)^2 \nu_t dt +
    \int_0^\tau \kappa_t u^2_t dt + c_\tau (X^u_\tau -
    \hat{\xi}^c_\tau)^2 \right],
\end{equation}
where the limes superior is taken over all sequences of stopping times
$(\tau^n)_{n=1,2,\ldots}$ converging to terminal time $T$ strictly
from below. Introducing the set of admissible controls
\begin{equation} \label{eq:admissibility}
\cU^c = \left\{ u \in \cU \text{ satisfying } J^c(u) < + \infty \right\}
\end{equation}
as in~\eqref{eq:policies}, we will solve completely the
auxiliary optimization problem
\begin{equation} \label{eq:problem}
J^c(u) \rightarrow \min_{u \in \cU^c}
\end{equation}
in the next section.

\subsection{Explicit solution to the auxiliary problem}
\label{subsec:mainresult}

As it turns out, the optimal control to our auxiliary stochastic LQ
problem in~\eqref{eq:problem} and its corresponding optimal value are
explicitly computable and fully characterized by the processes $c$ and
$\hat{\xi}^c$. In terms of these, we can also characterize when the
set of admissible controls $\cU^c$ defined in \eqref{eq:admissibility}
is nonempty. In fact, it follows from our analysis below that
$\cU^c \neq \varnothing$ if and only if
\begin{equation} \label{ass:integrability} 
  \EE\left[ \int_0^T (\xi_t -
    \hat{\xi}^c_t)^2 \nu_t dt \right] < + \infty \quad \text{and} \quad
  \EE\left[ \int_{[0,T)} c_t d[\hat{\xi}^c]_t \right] < + \infty,
\end{equation}
where $[\hat{\xi}^c]$ denotes the quadratic variation process of the
semimartingale $\hat{\xi}^c$ of~\eqref{eq:defoptsignal}. In
particular, \eqref{ass:integrability} is necessary and sufficient for
well-posedness of the LQ problem in \eqref{eq:problem}:

\begin{Theorem} \label{thm:main} 
  Let \eqref{eq:condkappanu}, \eqref{eq:condxi}, and \eqref{ass:etakappa}
  hold true. In addition, suppose that $c$ follows the Riccati
  dynamics~\eqref{eq:BSRDE} with terminal condition~\eqref{eq:BSRDEtc}
  and satisfies the integrability conditions~\eqref{eq:BSRDEintcond2}
  and~\eqref{ass:Xi}.

  Then we have $\cU^c \neq \varnothing$ if and only if
  \eqref{ass:integrability} is satisfied. In this case, the optimal
  control $\hat{u}^c \in \cU^c$ for the auxiliary problem
  \eqref{eq:problem} with controlled process
  $\hat{X}^c_{\cdot} \set X^{\hat{u}^c}_{\cdot}$ is given by the
  feedback law
  \begin{equation} \label{eq:optimalcontrol} 
    \hat{u}^c_t =
    \frac{c_t}{\kappa_t} \left(\hat{\xi}^c_t - \hat{X}^c_t \right) \quad
    (0 \leq t < T),
  \end{equation}
  and the minimal costs are
  \begin{align}
    J^c(\hat{u}^c) = 
    c_0 ( x - \hat{\xi}^c_0)^2
    + \EE \left[ \int_0^T (\xi_t - \hat{\xi}^c_t)^2 \nu_t dt \right]
    + \EE \left[ \int_{[0,T)} c_t d[\hat{\xi}^c
    ]_t \right]. \label{eq:minimalcosts} 
  \end{align}
\end{Theorem}

The proof of Theorem \ref{thm:main} is deferred to Section
\ref{sec:proofs} below. Observe that the feedback law of the optimal
control in \eqref{eq:optimalcontrol} prescribes a reversion towards
the optimal signal process $\hat{\xi}^c_t$ rather than towards the
current target position~$\xi_t$. The reversion speed is controlled by
the ratio $c_t/\kappa_t$. In particular, on the ``blow-up'' set
$\{ \eta = + \infty \}$ the optimizer reverts with stronger and
stronger urgency towards the optimal signal $\hat{\xi}^c$ and hence to
the ultimate target position $\Xi_T$ due to \eqref{eq:limitos}. This
result generalizes the insights from the constant coefficient case
with almost sure terminal state constraint which are presented in
\citet{BankSonerVoss:16}.

Under the integrability conditions \eqref{ass:integrability}, the
optimal costs $J^c(\hat{u}^c)$ in \eqref{eq:minimalcosts} of the
optimizer $\hat{u}^c$ in \eqref{eq:optimalcontrol} are obviously
finite. Actually, they nicely separate into three intuitively
appealing terms making transparent how the regularity and
predictability of the targets $\xi$ and $\Xi_T$ determine the
auxiliary problem's optimal value. The first term represents the costs
due to a possibly suboptimal initial position $x$. The second term
shows how the regularity of the target process~$\xi$ feeds into the
overall costs: Targets which are poorly approximated by the optimal
signal process $\hat{\xi}^c$ in the $L^2(\PP \otimes \nu_t dt)$-sense
produce higher costs. Finally, the third term reveals the importance
of the optimal signal's quadratic variation process
$[\hat{\xi}^c]$. Referring to the definition of $\hat{\xi}^c$ in
\eqref{eq:defoptsignal} (cf. also the representation in
\eqref{eq:repxi} below), the quadratic variation~$[\hat{\xi}^c]$ can
be viewed as a measure for the strength of the fluctuations in the
assessment of the average future target positions of $\xi$, the
terminal position $\Xi_T$ and the random variable $L_T$ which involves
the outcome of the penalization parameter $\eta$ at time $T$. In this
sense, the second integrability condition in~\eqref{ass:integrability}
can be interpreted as encoding a condition on the predictability of
the final stochastic target position $\Xi_T$ as well as the random
penalization parameter $\eta$. Loosely speaking, it ensures that the
outcome of the final position $\Xi_T$ as well as the ``blow-up'' event
$\{ \eta = + \infty \}$ on which $\Xi_T$ has to be matched by controls
in $\cU^c$ are not allowed to come as ``too big a surprise'' at final
time $T$; see also our discussion in Section~\ref{subsec:BSV} below.

\begin{Remark}[Interpretation of the optimal signal] \label{rem:osp}
  Let us present a way to interpret our optimal signal process
  $\hat{\xi}$ defined in~\eqref{eq:defoptsignal}. For ease of
  presentation and to avoid unnecessary technicalities, let us assume
  here that the convergence in~\eqref{eq:defLlimit} also holds in
  $L^1(\PP)$, that $\EE[L_T] > 0$ and that
  $0 < \nu \in L^1(\PP \otimes dt)$ (these assumptions merely simplify
  the justification of the representation in~\eqref{eq:kernel} below;
  cf. Lemma~\ref{lem:remark} in Section~\ref{sec:proofs}). Then, by
  defining the \emph{weight process} $(w_t)_{0 \leq t < T}$ via
    \begin{equation} \label{eq:weight} 
      w_t \set \frac{\EE[L_T \vert
        \cF_t]}{L_t} \qquad (0 \leq t < T)
    \end{equation}
    as well as the measure $\QQ \ll \PP$ on $(\Omega,\cF_T)$ via
    \begin{equation*} \label{eq:density} \frac{d\QQ}{d\PP} \set
      \frac{L_T}{\EE[L_T]},
    \end{equation*}
    we may write
    \begin{align}
      \hat{\xi}_t & = \frac{1}{L_t}
                    \EE\left[ \Xi_T L_T + \int_t^T \xi_r e^{-\int_0^r
        \frac{c_u}{\kappa_u} du} \nu_r dr \, \bigg\vert \, \cF_t
    \right] \nonumber \\
                  & = w_t \,
                    \EE_\QQ[\Xi_T \vert \cF_t] + (1-w_t) \, \EE \left[ \int_t^T \xi_r
                    \frac{e^{-\int_t^r \frac{c_u}{\kappa_u} du}}{(1-w_t)c_t}
                    \nu_r dr \bigg\vert \mathcal{F}_t \right] \label{eq:repxi} 
    \end{align}
    for all $0 \leq t < T$. Recall that the process
    $(L_t)_{0 \leq t < T}$ is a strictly positive supermartingale by
    virtue of Lemma \ref{lem:L}. Consequently, the weight process
    satisfies
    \begin{equation*}
      0 \leq w_t < 1 \quad \PP\text{-a.s. for all } 0 \leq t < T, 
    \end{equation*}
    where the strict inequality follows from Lemma~\ref{lem:remark}
    below because we assumed $\nu > 0$ here for simplicity. Moreover,
    the same lemma gives the identity
    \begin{equation}
      \EE \left[ \int_t^T \frac{e^{-\int_t^r \frac{c_u}{\kappa_u}du}}{(1-w_t)c_t}
        \nu_r dr \bigg\vert \mathcal{F}_t \right] = 1 \quad d\PP
      \otimes dt\text{-a.e. on } \Omega \times [0,T). \label{eq:kernel}
    \end{equation}
    That is, loosely speaking, the optimal signal process $\hat{\xi}$
    in \eqref{eq:repxi} is a convex combination of a weighted average
    of expected future target positions of $\xi$ and the expected
    terminal position~$\Xi_T$, computed under the auxiliary measure
    $\QQ$. The weight shifts gradually towards the ultimate target
    position $\Xi_T$ as $t \uparrow T$, provided that $L_T >
    0$. Indeed, by definition of the weight process
    in~\eqref{eq:weight}, martingale convergence theorem and the
    convergence of the process $L$ in Lemma \ref{lem:L}, we have
    \begin{equation*}
      \exists \, \lim_{t \uparrow T} w_t = 1 \quad \text{on the set }
      \{ L_T > 0 \}.
    \end{equation*}
\end{Remark}


\section{Discussion and illustration}
\label{sec:illustration}

Let us return to the initial stochastic LQ
problem~\eqref{eq:originalLQP} with target
functional~\eqref{eq:originalobjective} and stochastic terminal state
constraint~\eqref{eq:terminalconstraint} and discuss how it relates to our
auxiliary LQ problem~\eqref{eq:problem}. Observe that, for the latter,
we tackle and resolve the delicate partial terminal state constraint
$X^u_T = \Xi_T$ on $\{ \eta = + \infty \}$ incorporated in the set of
admissible policies $\cU^\Xi$ in~\eqref{def:originalpolicies} by
performing a \emph{truncation in time} in the auxiliary objective
functional $J^c$ in~\eqref{eq:defobjective}. Specifically, we replace
the original target functional $J^\eta$ of
problem~\eqref{eq:originalLQP} by a properly chosen limit of
stochastic LQ target functionals with strictly shorter time horizon
$\tau < T$ at which we impose a finite terminal penalization term
$c_{\tau} (X^u_\tau - \hat{\xi}^c_\tau)^2$. In fact, the optimal
signal process $\hat{\xi}^c$ turns out to be the proper key ingredient
for choosing these penalizations in a \emph{time consistent} manner;
see Remark~\ref{rem:timeconsistent} below. Moreover, in light of
$\liminf_{t \uparrow T} c_t \geq \eta$ in~\eqref{eq:BSRDEtc} and
$\lim_{t \uparrow T} \hat{\xi}^c_t = \Xi_T$ on
$\{0 < \eta \leq +\infty\}$ in \eqref{eq:limitos}, for any $\tau < T$
the penalty $c_{\tau} (X^u_\tau - \hat{\xi}^c_\tau)^2$ can be viewed
as a proxy of the terminal penalty
$\eta 1_{\{0 \leq \eta < +\infty \}} (X^u_T - \Xi_T)^2$ in
$J^\eta$. Indeed, appealing to Fatou's Lemma, monotone convergence as
well as \eqref{eq:BSRDEtc} and \eqref{eq:limitos}, we readily obtain
the following:

\begin{Lemma} \label{lem:dominate} It holds that $J^c(u) \geq J^\eta(u)$
  for all $u \in \cU^c$ and all processes $c$
  satisfying~\eqref{eq:BSRDE}, \eqref{eq:BSRDEtc} and~\eqref{ass:Xi}.
  In particular, we have
  \begin{equation*}
    X^u_T = \Xi_T \quad \text{ on the set } \{ \eta = + \infty \}
    \text{ for all } u \in \cU^c,
  \end{equation*}
  that is, $\cU^c \subseteq \cU^{\Xi}$. \qed
\end{Lemma}

In light of this lemma, it appears very natural for our auxiliary LQ
problem in~\eqref{eq:problem} to consider parameter processes $c$
which are \emph{minimal} supersolutions. In fact, this motivates our
conjecture~\eqref{eq:conjecture} that
\begin{equation*}
  \argmin_{\cU^{\Xi}} J^\eta = \argmin_{\cU^{c^{\min}}} J^{c^{\min}}
\end{equation*}
holds true for the minimal supersolution $c^{\min}$ to the BSRDE
in~\eqref{eq:BSRDE} with terminal condition~\eqref{eq:BSRDEtc}. In the
following paragraphs of this section, we provide evidence for the
validity of this conjecture. Specifically, we will show how our
approach via the auxiliary LQ problem in~\eqref{eq:problem} with
$c^{\min}$ allows us to recover existing results in the literature to
specific variants of the stochastic LQ problem with stochastic
terminal state constraint posed in~\eqref{eq:originalLQP}. We will
also discuss possible approaches to prove the
conjecture~\eqref{eq:conjecture} based on the insights from
Section~\ref{subsec:mainresult}; see the end of
Section~\ref{subsec:KP} and also Remark~\ref{rem:timeconsistent} in
Section~\ref{sec:proofs}.

\subsection{Bounded coefficients}
\label{subsec:KT}

In case where $\eta$ is bounded along with the processes
$(\nu_t)_{0 \leq t \leq T}$, $(\kappa_t)_{0 \leq t \leq T}$ and
$(\xi)_{0 \leq t \leq T}$, our conjecture in~\eqref{eq:conjecture}
holds true. Indeed, under this conditions and within a Brownian
framework,~\citet{KohlmannTang:02} provide existence and uniqueness of
a (minimal supersolution) $c^{\min}$ to the stochastic Riccati
equation~\eqref{eq:BSRDE} such that
$\lim_{t \uparrow T} c^{\min}_t = \eta$ holds true $\PP$-a.s. They
show that the optimal control $\hat{u}^{c^{\min}}$
in~\eqref{eq:optimalcontrol} from our Theorem~\ref{thm:main} solves
the LQ problem in~\eqref{eq:originalLQP} with objective functional
$J^\eta$ (over the set of unconstrained policies~$\cU^{\Xi}$, recall
Remark~\ref{rem:intro}, 1.)); see~\citet{KohlmannTang:02}, Theorem
5.2. Obviously, our necessary and sufficient integrability conditions
stated in~\eqref{ass:integrability} are satisfied in this case. Note,
though, that in~\cite{KohlmannTang:02}, Section~5.1, the optimal
control $\hat{u}^{c^{\min}}$ is characterized in terms of both the
process $c^{\min}$ and the solution process $b$ to the linear BSDE
\begin{equation} \label{eq:linBSDEKT} 
  db_t = \left(
    \frac{c^{\min}_t}{\kappa_t} b_t - \nu_t \xi_t \right) dt + dM_t \quad \text{on }
  [0,T] \text{ with } b_T = \eta \Xi_T,
\end{equation}
with some c\`adl\`ag (local) martingale $(M_t)_{0 \leq t \leq
  T}$. More precisely, the optimal control is
  described by the feedback law
\begin{equation*}
\hat{u}^{c^{\min}}_t = -\frac{1}{\kappa_t} \left( c^{\min}_t
    \hat{X}^{c^{\min}}_t - b_t \right) = \frac{c^{\min}_t}{\kappa_t}
  \left( \frac{b_t}{c^{\min}_t}-
    \hat{X}^{c^{\min}}_t  \right) \quad (0 \leq t \leq T).
\end{equation*}
That is, in that setting without terminal constraints, our signal
process $\hat{\xi}^{c^{\min}}$ of~\eqref{eq:defoptsignal} coincides
with the ratio $b/c^{\min}$ and so the solution $b$ to the linear BSDE
is an equivalent substitute for this signal process. In contrast, in
case where $\{\eta = +\infty, \Xi_T \not=0\}$ has positive
probability, the terminal condition $b_T=\eta \Xi_T$ becomes
problematic in the sense that the linear BSDE loses all information on
$\Xi_T$ except its sign. Our signal process $\hat{\xi}^{c^{\min}}$,
however, still makes sense in this rather natural case. Note that
$c^{\min}\hat{\xi}^{c^{\min}}$ still satisfies the linear BSDE
dynamics in~\eqref{eq:linBSDEKT} on $[0,T)$ (see
equation~\eqref{eq:dymcxi3} below) but this product may not have a
sensible terminal value on $\{\eta=0\}
\cup\{\eta=+\infty\}$. Fortunately, as our analysis shows the optimal
signal process always makes sense when needed. In particular its
possible lack of a terminal value on $\{\eta=0\}$ is without harm for
our approach to the optimization problem. It thus can be viewed as a
convenient substitute for the no longer operative linear BSDE above.

\subsection{Constant coefficients}
\label{subsec:BSV}

In case of constant coefficients $\nu_t \equiv \nu \in \RR_+$,
$\kappa_t \equiv \kappa \in \RR_+$ and $\eta \in [0,+\infty]$ the
stochastic Riccati differential equation in~\eqref{eq:BSRDE} boils
down to a deterministic \emph{ordinary Riccati differential equation}
on $[0,T]$ of the form
\begin{equation*}
  c'_t = \frac{c^2_t}{\kappa} - \nu\quad
  \text{ subject to }
  c_T = \eta 
\end{equation*}
with explicitly available deterministic (minimal super-)solutions
\begin{equation} \label{eq:solRODE}
  c^{\min}_t =
  \begin{cases}
    \sqrt{\nu \kappa} \, \frac{\sqrt{\nu \kappa} \sinh\left(
        \sqrt{\nu/\kappa}\,(T-t) \right) + \eta \cosh\left(
        \sqrt{\nu/\kappa}\,(T-t) \right)}{\eta \sinh\left(
        \sqrt{\nu/\kappa}\,(T-t) \right) + \sqrt{\nu \kappa}
      \cosh\left( \sqrt{\nu/\kappa}\,(T-t) \right)} & 0 \leq \eta
    < + \infty \\
    \sqrt{\nu \kappa} \coth(\sqrt{\nu} (T-t)/\sqrt{\kappa}), & \eta =
    + \infty
  \end{cases},
\end{equation}
for all $0 \leq t < T$. As a consequence, the process $L$ given
in~\eqref{eq:defL} is also just deterministic and the optimal signal
process $\hat{\xi}^{c^{\min}}$ in~\eqref{eq:defoptsignal} can be
computed explicitly (up to the conditional expectation). Again, our
conjecture in~\eqref{eq:conjecture} holds true. Indeed, our optimal
control $\hat{u}^{c^{\min}}$ from~\eqref{eq:optimalcontrol} provided
in Theorem~\ref{thm:main} coincides with the optimal solution of the
stochastic LQ problem in~\eqref{eq:originalLQP} with objective
functional $J^0$ and $J^{\infty}$, respectively, derived
in~\citet{BankSonerVoss:16}, Theorems~3.1 and~3.2. Therein, our first
integrability condition in~\eqref{ass:integrability} is satisfied as
soon as the target process $\xi$ belongs to $L^2(\PP \otimes dt)$ and
$\Xi_T \in L^2(\PP,\cF_{T-})$. The second integrability condition
in~\eqref{ass:integrability} simplifies to a condition on the terminal
position $\Xi_T$ which is equivalent to
\begin{equation*} 
\int_0^T \frac{\EE[(\Xi_T  - \EE[\Xi_T \vert \cF_s])^2]}{(T-s)^2} ds
< \infty;
\end{equation*}
see Remark 2.1 and Lemma 5.4 in \cite{BankSonerVoss:16}. It thus
reveals that the ultimate target position $\Xi_T$ has to become known
``fast enough'' for the optimally controlled process
$\hat{X}^{c^{\min}}$ in order to reach it at terminal time $T$ with
finite expected costs; c.f. also~\citet{AnkirchnerKruse:15} who
confine themselves to stochastic terminal state constraints of the
form $\Xi_T = \int_0^T \lambda_t dt$ for some progressively measurable
and suitably integrable process $(\lambda_t)_{0 \leq t \leq T}$ which
are gradually revealed as $t \uparrow T$. Related results of this
nature are also provided in~\citet{LueYongZhang:12}. 

For the general case with stochastic coefficients
$\nu = (\nu_t)_{0 \leq t \leq T}$,
$\kappa = (\kappa_t)_{0 \leq t \leq T}$ and random
$\eta \in [0,\infty]$, similar effects are to be expected concerning
the final target position $\Xi_T$ \emph{and} the ``blow up'' event
$\{ \eta = + \infty\}$. As, in general, all these coefficients can be
rather intricately intertwined among each other, it seems difficult to
give conditions on these that ensure $\cU^\Xi \neq \varnothing$ and
are more succinct than our conditions in~\eqref{ass:integrability}.

\subsection{Special case: Vanishing targets}
\label{subsec:KP}

In the special case $\xi \equiv \Xi_T \equiv 0$ $\PP$-a.s., where
obviously $\hat{\xi} \equiv 0$ and the integrability conditions
in~\eqref{ass:integrability} hold trivially, our conjecture
in~\eqref{eq:conjecture} holds true as
well. Indeed,~\citet{KrusePopier:16_1} derive under sufficient
integrability conditions on $(\kappa_t)_{0 \leq t \leq T}$ and
$(\nu_t)_{0 \leq t \leq T}$ existence of a minimal supersolution
$c^{\min}$ to the Riccati BSDE in~\eqref{eq:BSRDE} with terminal
condition $\liminf_{t \uparrow T} \geq \eta \in [0,+\infty]$
(recall~\eqref{eq:BSRDEtc}). The minimal supersolution $c^{\min}$ is
constructed via the monotone limit
$c_t^{\min} \set \lim_{n \uparrow \infty} c_t^{(n)}$ for all
$t \in [0,T)$, where $c^{(n)}$ denotes the unique (minimal
super-)solution with Riccati dynamics~\eqref{eq:BSRDE} satisfying the
terminal condition $\lim_{t \uparrow T} c^{(n)}_T = \eta \wedge n$ for
some constant $n > 0$. They show that the optimal control
$\hat{u}^{c^{\min}}$ with state process~\eqref{eq:defX} to the
stochastic LQ problem in~\eqref{eq:originalLQP} with
$\xi \equiv \Xi_T \equiv 0$ is given as in~\eqref{eq:optimalcontrol} of
our Theorem~\ref{thm:main}; see \cite{KrusePopier:16_1},
Theorem~3. That is, since $\hat{\xi}^{c^{\min}} \equiv 0$, the optimal
control with controlled process
$\hat{X}^{c^{\min}} \set X^{\hat{u}^{c^{\min}}}$ is simply given by
\begin{equation} \label{eq:optcontKP} 
  \hat{u}^{c^{\min}}_t =
  -\frac{c^{\min}_t}{\kappa_t} X^{\hat{u}^{c^{\min}}}_t = -\frac{x
    L_t}{\kappa_t} \quad (0 \leq t \leq T)
\end{equation}
and the corresponding optimal costs
in~\eqref{eq:minimalcosts} simplify dramatically to
\begin{equation} \label{eq:optvalKP}
  J^{c^{\min}}(\hat{u}^{c^{\min}}) = c^{\min}_0 x^2.
\end{equation}
In fact, in order to tackle the partial state constraint $\Xi_T = 0$ on the
set $\{ \eta = + \infty \}$,~\citet{KrusePopier:16_1} proceed via a
\emph{truncation in space}. Specifically, they introduce a family of
unconstrained variants of problem~\eqref{eq:originalLQP} (with
$\xi \equiv \Xi_T \equiv 0$) with objective functionals
\begin{equation} \label{eq:objectivetruncspace}
  J^{(n)}(u) \set \EE \left[
    \int_0^T (X^u_t)^2 \nu_t dt + \int_0^T \kappa_t u^2_t dt +
    (\eta \wedge n) (X^u_T)^2 \right],
\end{equation}
where the random penalization parameter $\eta$ is replaced by
truncated versions $\eta \wedge n$. Then the corresponding optimal
controls $\hat{u}^{(n)}_t=-c^{(n)}_t X^{\hat{u}^{(n)}}_t/\kappa_t$ and
the corresponding optimal costs
$J^{(n)}(\hat{u}^{(n)}) = c_0^{(n)}x^2$ clearly satisfy
\begin{equation} \label{eq:KPequality}
\begin{aligned}
  J^{\eta}(\hat{u}^\eta) \leq & \; J^{c^{\min}}(\hat{u}^{c^{\min}}) =
  c^{\min}_0 x^2 = \lim_{n
    \uparrow \infty} c_0^{(n)} x_0^2 \\
= & \; \lim_{n
    \uparrow \infty} J^{(n)} (\hat{u}^{(n)}) = \lim_{n \uparrow \infty}
  J^{c^{(n)}}(\hat{u}^{c^{(n)}}) \leq J^{\eta}(\hat{u}^\eta),
\end{aligned}
\end{equation}
where $\hat{u}^\eta$ denotes the optimizer of
problem~\eqref{eq:originalLQP} (with $\xi \equiv \Xi_T \equiv 0$). It
follows that equality holds everywhere and, by uniqueness of
optimizers, $\hat{u}^\eta = \hat{u}^{c^{\min}}$ as conjectured
  in~\eqref{eq:conjecture}.

For the general case $\xi \neq 0$ and $\Xi_T \neq 0$, one could
likewise introduce as above in~\eqref{eq:objectivetruncspace} a family
of unconstrained variants of problem~\eqref{eq:originalLQP} with
objective functionals
\begin{equation*}
  J^{(n)}(u) = \EE \left[
    \int_0^T (X^u_t - \xi_t)^2 \nu_t dt + \int_0^T \kappa_t u^2_t dt +
    (\eta \wedge n) (X^u_T - \Xi_T)^2 \right].
\end{equation*}
Recall from the discussion in Section~\ref{subsec:KT} that this
stochastic LQ problem is fully characterized by the solution processes
$c^{(n)}$ and $b^{(n)}$ satisfying the Riccati BSDE in
\eqref{eq:BSRDE} and the linear BSDE in~\eqref{eq:linBSDEKT} with
terminal conditions $c^{(n)}_T = \eta \wedge n$ and
$b^{(n)}_T = (\eta \wedge n) \Xi_T$, respectively. In addition, under
sufficient conditions (e.g. boundedness as discussed in
Section~\ref{subsec:KT}) which guarantee~\eqref{ass:integrability} as
well as the convergence
\begin{equation*}
\limsup_{\tau
  \uparrow T} \E\left[ c^{(n)}_{\tau} (X_\tau^{\hat{u}^{c^{(n)}}} -
\xi_{\tau}^{c^{(n)}} )^2 \right] = \E \left[ (\eta \wedge n) (X_T^{\hat{u}^{c^{(n)}}} -
\Xi_{T} )^2 \right],
\end{equation*}
our Theorem~\ref{thm:main} applies in this
context and it holds that
\begin{equation} \label{eq:limitJ}
\begin{aligned}
  J^{(n)}(\hat{u}^{(n)}) = & \; 
  c^{(n)}_0 ( x - \hat{\xi}^{c^{(n)}}_0)^2 \\
  & + \EE \left[ \int_0^T (\xi_t - \hat{\xi}^{c^{(n)}}_t)^2 \nu_t dt \right]
  + \EE \left[ \int_{[0,T)} c_t^{(n)} d[\hat{\xi}^{c^{(n)}}
  ]_t \right] 
\end{aligned}
\end{equation}
for the optimal control $\hat{u}^{(n)}$.  As
in~\citet{KrusePopier:16_1}, one could then try to pass to the limit
$n \uparrow \infty$. However, passing to the limit is not as
straightforward in \eqref{eq:limitJ} as it is in \eqref{eq:KPequality}
where we relied heavily on $\xi \equiv 0$, $\Xi_T \equiv 0$.  Indeed,
for convergence of~\eqref{eq:limitJ}, a suitable convergence of our
signal processes $\hat{\xi}^{c^{(n)}}$ would be required which seems
to be out of reach with the current knowledge of singular BSRDEs and
so a full proof of our conjecture~\eqref{eq:conjecture} by this
approach has to be left for future research.

\section{Proofs}
\label{sec:proofs}

Throughout this section we work under the assumptions of our main
result, Theorem \ref{thm:main}. Its verification relies on a
completion of squares argument similar to \citet{KohlmannTang:02}
(cf. also \citet{YongZhou:99} for this method in solving LQ
problems). The following lemma summarizes the key identity for our
verification and illustrates again the usefulness of our signal
process $\hat{\xi}$.

\begin{Lemma} \label{lem:mastereq} Suppose the assumptions of
  Theorem~\ref{thm:main} hold true. Then for all progressively
  measurable, $\PP$-a.s. locally $L^2([0,T),\kappa_t dt)$-integrable
  processes $u$, the cost process
  \begin{equation*} 
    C_t(u) \set \int_0^t (X^u_s - \xi_s)^2 \nu_s ds 
    + \int_0^t \kappa_s u^2_s ds + c_t
    (X^u_t - \hat{\xi}_t)^2 \quad  (0 \leq t < T)
  \end{equation*}
  is a nonnegative, c\`adl\`ag local submartingale. It allows for the
  decomposition
  \begin{equation} \label{eq:master}
    C_t(u) = c_0(x-\hat{\xi}_0)^2 + A_t(u) + M_t(u) 
    \quad  (0 \leq t < T), 
  \end{equation}
  where
  \begin{align}
    A_t(u) \set 
    & ~\int_0^t (\xi_s - \hat{\xi}_s)^2 \nu_s ds +
      \int_0^t c_s d[\hat{\xi}]_s \nonumber \\
    & + \int_0^t \kappa_s \left( u_s -
      \frac{c_s}{\kappa_s} \left( \hat{\xi}_s - X^u_s \right)
      \right)^2 ds \quad  (0 \leq t < T)\label{eq:A} \\
    \intertext{is a right continuous, nondecreasing, adapted process
    and where}
    M_t(u) \set 
    & ~ \int_0^t (\hat{\xi}^2_{s-} - (X^u_{s-})^2) dN_s +
      2 \int_0^t \frac{c_{s-}}{L_{s-}} (\hat{\xi}_{s-} -
      X^u_{s-}) d\tilde{M}_s \label{eq:M} \\
    \intertext{with}
    \tilde{M}_t \set &~ \EE \left[\Xi_T L_T + \int_0^T \xi_s e^{-\int_0^s
    \frac{c_u}{\kappa_u} du} \nu_s ds \, \bigg\vert\, \cF_t \right] \quad  (0 \leq t < T)
        \label{eq:MtY}
  \end{align}
  is a local martingale on $[0,T)$.
\end{Lemma}

\begin{proof}
  First, note that by~\eqref{eq:condkappanu} and
  $u \in L^2([0,T),\kappa_t dt)$ locally $\PP$-a.s., the cost process
  $(C_t(u))_{0 \leq t < T}$ in~\eqref{eq:master} is well defined along
  with $X^u$. Let us expand
  \begin{equation*}
    c_t ( X^u_t - \hat{\xi}_t)^2 = c_t (X^u_t)^2 - 2 X^u_t c_t
    \hat{\xi}_t + c_t \hat{\xi}^2_t \quad (0 \leq t < T)
  \end{equation*}
  and then apply It\^o's formula to each of the resulting three
  terms. This will be prepared by computing the dynamics of the
  processes $\hat{\xi}$, $c\hat{\xi}$ and $c\hat{\xi}^2$,
  respectively, in the following steps 1, 2 and 3. In step 4 we put
  everything together and derive our main identity \eqref{eq:master}.

  \emph{Step 1:} We start with computing the dynamics of our optimal
  signal process $(\hat{\xi}_t)_{0 \leq t < T}$ defined in
  \eqref{eq:defoptsignal}. For ease of notation, let us define the process
  \begin{equation*}
    Y_t  \set \int_0^t \xi_r e^{-\int_0^r \frac{c_u}{\kappa_u} du} \nu_r
    dr \quad (0 \leq t \leq T).
  \end{equation*}
  Observe that $Y_T \in L^1(\PP)$ due to \eqref{eq:condxi}.  Moreover,
  recall that $\Xi_T L_T \in L^1(\cF_{T-},\PP)$ by \eqref{ass:Xi} so
  that \eqref{eq:MtY} defines a c\`adl\`ag martingale on
  $[0,T]$. By the definition of $\hat{\xi}$ in
  \eqref{eq:defoptsignal}, we can now express $\hat{\xi}$ in terms of
  $Y$ and $\tilde{M}$ via
  \begin{equation}
    \hat{\xi}_t =  \frac{1}{L_t}
                    \EE\left[ \Xi_T L_T + \int_t^T \xi_r
                      e^{-\int_0^r \frac{c_u}{\kappa_u} du} \nu_r dr \, \bigg\vert \, \cF_t
                    \right]
    = \frac{1}{L_t} \left( \tilde{M}_t - Y_t \right) \label{eq:optsignal}
  \end{equation}
  for all $0 \leq t < T$. Next, recall the dynamics of $L$ on $[0,T)$
  in \eqref{eq:SDEL} and note that
  \begin{equation}
    \Delta L_t = - \frac{L_{t-}}{c_{t-}} \Delta N_t \quad \text{and} \quad
    [L]^c_t = \int_0^t \frac{L^2_{s-}}{c^2_{s-}} d[N]^c_s, \label{eq:jumpL}
  \end{equation}
  where $[L]^c$ and $[N]^c$ denote the path-by-path continuous parts
  of the quadratic variations of $[L]$ and $[N]$, respectively (cf.,
  e.g., \citet{Prot:04}, Chapter II.6, for more details). Hence,
  applying It\^o's formula as in, e.g., \cite{Prot:04}, Theorem II.32,
  we obtain
  \begin{align}
    \frac{1}{L}_t = & ~\frac{1}{L_0} - \int_0^t \frac{1}{L^2_{s-}} dL_s
                      + \int_0^t \frac{1}{L^3_{s-}} d[L]_s^c \nonumber \\
                    & + \sum_{s \leq t} \left( \frac{1}{L_s} - \frac{1}{L_{s-}} +
                      \frac{1}{L^2_{s-}} \Delta L_s \right). \label{eq:dyminvL1}
  \end{align}
  Using \eqref{eq:jumpL}, the summands in
  the sum in \eqref{eq:dyminvL1} above can be written as
  \begin{equation*}
    \frac{L_{s-} - L_s}{L_s L_{s-}} - \frac{\Delta
      N_s}{L_{s-}c_{s-}} = \frac{\Delta N_s}{c_{s-}}
    \frac{L_{s-} - L_s}{L_s L_{s-}} = \frac{(\Delta
      N_s)^2}{L_s c_{s-}^2} = \frac{(\Delta
      N_s)^2}{L_{s-} c_{s-} c_s},
  \end{equation*} 
  where we also used $\Delta c_s = - \Delta N_s$ and thus the identity
  $1/L_s = c_{s-}/(L_{s-} c_s)$ in the last equality. Hence,
  together with the dynamics of $L$ in \eqref{eq:SDEL} and $[L]^c$ in
  \eqref{eq:jumpL} we can rewrite \eqref{eq:dyminvL1} as
  \begin{align}
    \frac{1}{L}_t = & ~ \frac{1}{L_0} + \int_0^t
                      \frac{\nu_s}{L_{s-} c_{s-}} ds
                      + \int_0^t \frac{1}{L_{s-} c_{s-}} dN_s \nonumber \\
                    & +
                      \int_0^t \frac{1}{L_{s-} c^2_{s-}} d[N]^c_s + \sum_{s \leq t} \frac{(\Delta
      N_s)^2}{L_{s-} c_{s-} c_s}. \label{eq:dyminvL2}
  \end{align}
  Now, integrating by parts in \eqref{eq:optsignal} and then using the
  dynamics of $1/L$ in \eqref{eq:dyminvL2} gives us
  \begin{align}
    \hat{\xi}_t 
    = & ~ \hat{\xi}_0 + \int_0^t \frac{1}{L_{s-}}
        (d\tilde{M}_s - dY_s) + \int_0^t \hat{\xi}_{s-} L_{s-}
        d\left( \frac{1}{L_s} \right) + \left[ \frac{1}{L},
        \tilde{M} \right]_t \nonumber \\
    = & ~ \hat{\xi}_0 - \int_0^t
        (\xi_s -\hat{\xi}_{s-}) \frac{\nu_s}{c_{s-}} ds + \int_0^t
        \frac{1}{L_{s-}} d\tilde{M}_s + \int_0^t
        \frac{\hat{\xi}_{s-}}{c_{s-}} dN_s \nonumber \\
      & + \int_0^t
        \frac{\hat{\xi}_{s-}}{c^2_{s-}} d[N]^c_s +
        \sum_{s \leq t} \frac{\hat{\xi}_{s-}}{c_{s-} c_s } (\Delta
        N_s)^2 
        + \left[ \frac{1}{L},
        \tilde{M} \right]_t, \label{eq:dymxi1}
  \end{align}
  where the quadratic covariation can be computed as
  \begin{align}
    \left[ \frac{1}{L},
    \tilde{M} \right]_t = 
    & 
      ~\int_0^t \frac{1}{L_{s-}
      c_{s-}} d
      [\tilde{M},N]^c_s \nonumber \\
    & ~+
      \sum_{s \leq t} \left( \frac{\Delta \tilde{M}_s
      \Delta N_s}{L_{s-} c_{s-}}
      + \frac{(\Delta N_s)^2 \Delta \tilde{M}_s}{L_{s-} c_{s-} c_s }
      \right). \label{eq:dymxi2}
  \end{align}
  Collecting all the sums in \eqref{eq:dymxi1} together with those in
  \eqref{eq:dymxi2} yields
  \begin{align}
    & \sum_{s \leq t} \frac{\Delta N_s}{L_{s-} c_{s-} c_s} \left( c_s
      \Delta \tilde{M}_s + \Delta N_s \Delta \tilde{M}_s +
      \hat{\xi}_{s-} L_{s-} \Delta N_s \right) \nonumber \\
    & = \sum_{s \leq t} \frac{\Delta N_s}{L_{s-} c_{s-} c_s} \left(
      \hat{\xi}_s L_s c_{s-} - \hat{\xi}_{s-} L_{s-} c_s \right), \label{eq:dymxi3}
  \end{align}
  where we used the fact that $\Delta N_s = -\Delta c_s$ as well as
  \begin{equation}
    \Delta \tilde{M}_s = \tilde{M}_s - \tilde{M}_{s-} = \hat{\xi}_s L_s -
    \hat{\xi}_{s-} L_{s-} \label{eq:jumpMtilde}
  \end{equation}
  due to the representation of $\hat{\xi}$ in \eqref{eq:optsignal} and
  the continuity of $Y$. Plugging back \eqref{eq:dymxi3} into
  \eqref{eq:dymxi1} finally gives us
  \begin{align}
    \hat{\xi}_t = 
    & ~ \hat{\xi}_0 - \int_0^t
      (\xi_s -\hat{\xi}_{s-}) \frac{\nu_s}{c_{s-}} ds + \int_0^t
      \frac{1}{L_{s-}} d\tilde{M}_s + \int_0^t
      \frac{\hat{\xi}_{s-}}{c_{s-}} dN_s \nonumber \\
    & + \int_0^t
      \frac{\hat{\xi}_{s-}}{c^2_{s-}} d[N]^c_s 
      + \int_0^t \frac{1}{L_{s-}
      c_{s-}} d[\tilde{M},N]^c_s \nonumber \\
    & + \sum_{s \leq t} \frac{\Delta N_s}{L_{s-} c_{s-} c_s} \left(
      \hat{\xi}_s L_s c_{s-} - \hat{\xi}_{s-} L_{s-} c_s \right). \label{eq:dymxi4}
  \end{align}

  \emph{Step 2:} Let us now compute the dynamics of
  $c\hat{\xi}$. Again, integration by parts, together with the
  dynamics of $\hat{\xi}$ in \eqref{eq:dymxi4}, yields
  \begin{align}
    c_t \hat{\xi}_t = 
    & ~ c_0 \hat{\xi}_0 + \int_0^t c_{s-} d\hat{\xi}_s
      + \int_0^t \hat{\xi}_{s-} dc_s + \left[ c,
      \hat{\xi} \right]_t \nonumber \\
    = & ~ c_0 \hat{\xi}_0 - \int_0^t \xi_s \nu_s ds + \int_0^t
        \hat{\xi}_{s-} \frac{c_s^2}{\kappa_s} ds + \int_0^t
        \frac{c_{s-}}{L_{s-}} d\tilde{M}_s \nonumber \\
    & + \int_0^t \frac{\hat{\xi}_{s-}}{c_{s-}} d[N]^c_s +
      \int_0^t \frac{1}{L_{s-}} d[\tilde{M}, N]^c_s
      \nonumber \\
    & + \sum_{s \leq t} \frac{\Delta N_s}{L_{s-} c_s} \left(
      \hat{\xi}_s L_s c_{s-} - \hat{\xi}_{s-} L_{s-} c_s
      \right) + \left[ c,
      \hat{\xi} \right]_t. \label{eq:dymcxi1}
  \end{align}
  The quadratic covariation in \eqref{eq:dymcxi1} can be computed as
  \begin{align}
    \left[ c, \hat{\xi} \right]_t = 
    & - \int_0^t \frac{1}{L_{s-}} d[\tilde{M}, N]^c_s 
      - \int_0^t \frac{\hat{\xi}_{s-}}{c_{s-}} d[N]^c_s 
      \nonumber \\
    & - \sum_{s \leq t} \frac{\Delta N_s \Delta \tilde{M}_s}{L_{s-}}
      - \sum_{s \leq t} \frac{\hat{\xi}_{s-} (\Delta N_s)^2}{c_{s-}} \nonumber
    \\
    & - \sum_{s \leq t} \frac{(\Delta N_s)^2}{L_{s-} c_{s-} c_s} \left(
    \hat{\xi}_s L_s c_{s-} - \hat{\xi}_{s-} L_{s-} c_s \right). \label{eq:dymcxi2}
  \end{align}
  The sums of the jumps in the quadratic covariation in
  \eqref{eq:dymcxi2} can be rewritten (using again the identity in
  \eqref{eq:jumpMtilde} as well as the fact that $\Delta c_s = -
  \Delta N_s$) as
  \begin{align*}
    & - \sum_{s \leq t} \frac{\Delta N_s}{L_{s-} c_{s-} c_s} \left(
      \Delta \tilde{M}_s c_s c_{s-} + \hat{\xi}_{s-} \Delta N_s L_{s-}
      c_s + \Delta N_s
      \hat{\xi}_s L_s c_{s-} - \Delta N_s \hat{\xi}_{s-} L_{s-} c_s \right) \\
    & = - \sum_{s \leq t} \frac{\Delta N_s}{L_{s-} c_s} \left(
      \hat{\xi}_s L_s c_{s-} - \hat{\xi}_{s-} L_{s-} c_s \right).
  \end{align*}
  With this observation, plugging back the quadratic covariation in
  \eqref{eq:dymcxi2} into \eqref{eq:dymcxi1}, we simply get
  \begin{equation}
    c_t \hat{\xi}_t = c_0 \hat{\xi}_0 - \int_0^t \xi_s \nu_s ds + \int_0^t
    \hat{\xi}_{s-} \frac{c_s^2}{\kappa_s} ds + \int_0^t
    \frac{c_{s-}}{L_{s-}} d\tilde{M}_s. \label{eq:dymcxi3}
  \end{equation}

  \emph{Step 3:} Next, we compute the dynamics of
  $c\hat{\xi}^2$. Application of integration by parts together
  with the dynamics of $\hat{\xi}$ in \eqref{eq:dymxi4} yields
  \begin{align*}
    \hat{\xi}^2_t = & ~\hat{\xi}^2_0 + 2 \int_0^t \hat{\xi}_{s-}
                      d\hat{\xi}_s + [\hat{\xi}]_t \\
    = & ~\hat{\xi}^2_0 - 2 \int_0^t \hat{\xi}_{s-} (\xi_s - \hat{\xi}_{s-})
        \frac{\nu_s}{c_{s-}} ds + 2 \int_0^t
        \frac{\hat{\xi}_{s-}}{L_{s-}} d\tilde{M}_s + 2 \int_0^t
        \frac{\hat{\xi}^2_{s-}}{c_{s-}} dN_s \\
                    & + 2 \int_0^t
                      \frac{\hat{\xi}^2_{s-}}{c^2_{s-}} d[N]^c_s + 2
                      \int_0^t \frac{\hat{\xi}_{s-}}{L_{s-} c_{s-}} d[
                      \tilde{M},N]^c_s  \\
                    & + 2 \sum_{s \leq t} \frac{\hat{\xi}_{s-} \Delta N_s}{L_{s-} c_{s-}
                      c_s} \left( \hat{\xi}_s L_s c_{s-} - \hat{\xi}_{s-} L_{s-} c_s
                      \right) + [\hat{\xi}]_t.
  \end{align*}
  Consequently, using once more integration by parts, we
  obtain
  \begin{align}
    c_t \hat{\xi}^2_t = 
    & ~c_0 \hat{\xi}^2_0 + \int_0^t c_{s-} d\hat{\xi}^2_s
      + \int_0^t \hat{\xi}^2_{s-} dc_s +
      [c,\hat{\xi}^2]_t \nonumber \\
    = & ~c_0 \hat{\xi}^2_0 - 2 \int_0^t \hat{\xi}_{s-} (\xi_s - \hat{\xi}_{s-})
        \nu_s ds + 2 \int_0^t
        \frac{c_{s-} \hat{\xi}_{s-}}{L_{s-}} d\tilde{M}_s + 2 \int_0^t
        \hat{\xi}^2_{s-} dN_s \nonumber \\
    & + 2 \int_0^t
      \frac{\hat{\xi}^2_{s-}}{c_{s-}} d[N]^c_s + 2
      \int_0^t \frac{\hat{\xi}_{s-}}{L_{s-}} d[\tilde{M},N]^c_s  \nonumber \\
    & + 2 \sum_{s \leq t} \frac{\hat{\xi}_{s-} \Delta N_s}{L_{s-}
      c_s} \left( \hat{\xi}_s L_s c_{s-} - \hat{\xi}_{s-} L_{s-} c_s
      \right) + \int_0^t c_{s-} d[\hat{\xi}]_s \nonumber \\
    & + \int_0^t \hat{\xi}^2_{s-} \frac{c_s^2}{\kappa_s} ds - \int_0^t
      \hat{\xi}^2_{s-} \nu_s ds - \int_0^t
      \hat{\xi}^2_{s-} dN_s + [c,\hat{\xi}^2]_t.
                          \label{eq:dymcxisq1}
  \end{align}
  The final quadratic covariation in \eqref{eq:dymcxisq1} can be
  computed as
  \begin{align}
    [c,\hat{\xi}^2]_t = 
    & ~ -2 \int_0^t \frac{\hat{\xi}_{s-}}{L_{s-}} d[\tilde{M},N]^c_s 
      - 2 \sum_{s \leq
      t} \frac{\hat{\xi}_{s-}}{L_{s-}} \Delta
      \tilde{M}_s \Delta N_s
      \nonumber \\
    & - 2 \int_0^t \frac{\hat{\xi}^2_{s-}}{c_{s-}} d[N]^c_s - 2 \sum_{s \leq
      t} \frac{\hat{\xi}^2_{s-}}{c_{s-}} (\Delta
      N_s)^2 \nonumber \\
    & - 2  \sum_{s \leq t}
      \frac{\hat{\xi}_{s-} (\Delta N_s)^2}{L_{s-} c_{s-}
      c_s} \left( \hat{\xi}_s L_s c_{s-} - \hat{\xi}_{s-} L_{s-} c_s
      \right) + \int_0^t \Delta c_s d[\hat{\xi}]_s. \label{eq:dymcxisq2}
  \end{align}
  Observe that the sum of jumps in \eqref{eq:dymcxisq2} can
  be rewritten as
  \begin{align*}
    & -2 \sum_{s \leq t} \frac{\hat{\xi}_{s-} \Delta N_s}{L_{s-} c_{s-} c_s}
      \left( \Delta \tilde{M}_s c_s c_{s-} + \hat{\xi}_{s-} \Delta N_s
      c_{s} L_{s-} \right. \\
    & \hspace{96pt} \left. + \Delta N_s \hat{\xi}_s L_{s} c_{s-} -
      \hat{\xi}_{s-} L_{s-} c_{s} \Delta N_s \right) \\
    & = -2 \sum_{s \leq t} \frac{\hat{\xi}_{s-} \Delta N_s}{L_{s-}
      c_s} \left( \hat{\xi}_s L_s c_{s-} - \hat{\xi}_{s-} L_{s-} c_s
      \right),
  \end{align*}
  where we used once more the identity in \eqref{eq:jumpMtilde} and
  $\Delta c_s = -\Delta N_s$. With this observation, plugging back
  \eqref{eq:dymcxisq2} into \eqref{eq:dymcxisq1}, we finally obtain
  \begin{align}
    c_t \hat{\xi}^2_t
    = & ~c_0 \hat{\xi}^2_0 - 2 \int_0^t \hat{\xi}_{s-} \xi_s
        \nu_s ds + \int_0^t \hat{\xi}^2_{s-}
        \nu_s ds + 2 \int_0^t
        \frac{c_{s-} \hat{\xi}_{s-}}{L_{s-}} d\tilde{M}_s \nonumber \\
      & + \int_0^t
        \hat{\xi}^2_{s-} dN_s + \int_0^t c_s d[\hat{\xi}]_s + \int_0^t
        \hat{\xi}^2_{s-} \frac{c_s^2}{\kappa_s} ds.
                          \label{eq:dymcxisq3}
  \end{align}

  \emph{Step 4:} Let us now put together all the computations from the
  preceding steps. Specifically, let $u$ be a progressively
  measurable, $\PP$-a.s. locally $L^2([0,T),\kappa_t dt)$-integrable
  process with corresponding controlled process $X^u$. Due to our
  computations in \eqref{eq:dymcxi3} and \eqref{eq:dymcxisq3} as well
  as the fact that $X^u$ is continuous and of finite variation, we get
  for all $0 \leq t < T$ that
  \begin{align}
    & c_t ( X^u_t - \hat{\xi}_t)^2 = c_t (X^u_t)^2 - 2 X^u_t c_t
      \hat{\xi}_t + c_t \hat{\xi}^2_t \nonumber \\
    & = c_0(x-\hat{\xi}_0)^2 + \int_0^t c_s d[\hat{\xi}]_s -\int_0^t (X^u_{s})^2 \nu_s ds + 2 \int_0^t X^u_s
      \nu_s \xi_s ds \nonumber \\
    & \hspace{12pt} - 2 \int_0^t
      c_{s} u_s (\hat{\xi}_{s} - X^u_s) ds + \int_0^t
      \frac{c_s^2}{\kappa_s} (X^u_{s} - \hat{\xi}_{s})^2 ds - 2 \int_0^t \hat{\xi}_s \xi_s \nu_s ds +
      \int_0^t \hat{\xi}^2_s \nu_s ds \nonumber \\
    & \hspace{12pt} + \int_0^t (\hat{\xi}^2_{s-} - (X^u_{s-})^2) dN_s + 2
      \int_0^t \frac{c_{s-}
      }{L_{s-}} \left( \hat{\xi}_{s-} - X^u_{s-} \right) d\tilde{M}_s. \label{eq:C1}
  \end{align}
  Observe that the last two stochastic integrands sum up to $M_t(u)$
  defined in~\eqref{eq:M}. Furthermore, two completions of squares in
  the third line of \eqref{eq:C1} yield
  \begin{align}
    & c_t ( X^u_t - \hat{\xi}_t)^2 \nonumber \\ 
    & = c_0(x-\hat{\xi}_0)^2 + \int_0^t
      c_s d[\hat{\xi}]_s -\int_0^t (X^u_{s})^2 \nu_s ds + 2 \int_0^t X^u_s
      \nu_s \xi_s ds \nonumber \\
    & \hspace{12pt} + \int_0^t \kappa_s \left( u_s -
      \frac{c_s}{\kappa_s} \left( \hat{\xi}_s - X^u_s \right)
      \right)^2 ds + \int_0^t
      (\xi_s - \hat{\xi}_s)^2
      \nu_s ds \nonumber \\
    & \hspace{12pt}  - \int_0^t
      \kappa_s u^2_s ds - \int_0^t \xi^2_s \nu_s ds + M_t(u) \nonumber \\
    &=  ~ c_0(x-\hat{\xi}_0)^2 + \int_0^t
      c_s d[\hat{\xi}]_s - \int_0^t (X^u_s - \xi_s)^2 \nu_s ds \nonumber \\
    & \hspace{12pt} + \int_0^t \kappa_s \left( u_s -
      \frac{c_s}{\kappa_s} \left( \hat{\xi}_s - X^u_s \right)
      \right)^2 ds + \int_0^t
      (\xi_s - \hat{\xi}_s)^2
      \nu_s ds \nonumber \\
    & \hspace{12pt} - \int_0^t
      \kappa_s u^2_s ds +
      M_t(u)  \nonumber 
  \end{align}
  Consequently, we can write
  \begin{align}
    0 \leq C_t(u) = & ~\int_0^t (X^u_s - \xi_s)^2 \nu_s ds 
                      + \int_0^t \kappa_s u^2_s ds + c_t
                      (X^u_t - \hat{\xi}_t)^2 \nonumber \\
    = & ~ c_0(x-\hat{\xi}_0)^2 + \int_0^t
        c_s d[\hat{\xi}]_s + \int_0^t
        (\xi_s - \hat{\xi}_s)^2
        \nu_s ds \nonumber \\
                    & ~ + \int_0^t \kappa_s \left( u_s -
                      \frac{c_s}{\kappa_s} \left( \hat{\xi}_s - X^u_s \right)
                      \right)^2 ds + M_t(u) \nonumber \\
    = & ~ c_0(x-\hat{\xi}_0)^2 + A_t(u) + M_t(u) \quad (0 \leq t < T) 
        \label{eqC3}
  \end{align}
  with $(A_t(u))_{0 \leq t < T}$ as defined in \eqref{eq:A}. Finally,
  observe that the process $(A_t(u))_{0 \leq t < T}$ is a right
  continuous, nondecreasing, adapted process and that
  $(M_t(u))_{0 \leq t < T}$ is a c\`adl\`ag local martingale because
  $\tilde{M}$ and $N$ are local martingales on $[0,T)$ and all
  integrands in \eqref{eq:M} are left continuous (cf., e.g.,
  \citet{Prot:04}, Theorem III.33). Consequently, we have that
  $(C_t(u))_{0 \leq t < T}$ is a nonnegative, c\`adl\`ag local
  submartingale.
\end{proof}

We are now ready to give the proof of our main Theorem \ref{thm:main}:
\medskip

\noindent\textbf{Proof of Theorem~\ref{thm:main}:} First, let
us assume that $\cU^c \neq \varnothing$. For any $u \in \cU^c$ we can
consider the corresponding cost process
$C_t(u) = c_0(x-\hat{\xi}_0)^2 + A_t(u) + M_t(u)$, $0 \leq t < T$, as
in~\eqref{eq:master} of Lemma~\ref{lem:mastereq} above. Let
$(\tau^n)_{n = 1,2,\ldots}$ be a localizing sequence of stopping times
for the local martingale $(M_t(u))_{0 \leq t < T}$ such that
$\tau^n \uparrow T$ $\PP$-a.s. strictly from below as
$n \rightarrow \infty$ and $(M_{t \wedge \tau^n}(u))_{0 \leq t < T}$
is a uniformly integrable martingale for each $n$ (cf., e.g.,
\citet{Prot:04}, Chapter I.6, for more details). Then it holds by
definition of our performance functional~$J$ in
\eqref{eq:defobjective} that
\begin{align}
  \infty > J(u)  
  & \set \limsup_{\tau \uparrow T} \EE[C_{\tau}(u)] \nonumber \\
  & \geq ~c_0(x-\hat{\xi}_0)^2 + \limsup_{n \rightarrow
    \infty} \left\{ \EE[A_{\tau^n}(u)] + \EE[M_{\tau^n}
    (u)] \right\} \nonumber \\
    & = ~c_0(x-\hat{\xi}_0)^2 \nonumber\\
    & \hspace{18pt} + \EE \left[ \int_0^T (\xi_s - \hat{\xi}_s)^2 \nu_s ds +
      \int_0^T c_s d[\hat{\xi}]_s \right. \nonumber\\ 
  & \hspace{95pt} \left. + \int_0^T \kappa_s \left( u_s -
    \frac{c_s}{\kappa_s} \left( \hat{\xi}_s - X^u_s \right)
    \right)^2 ds\right] \nonumber \\
  & \geq ~c_0(x-\hat{\xi}_0)^2 + \EE\left[ \int_0^T
    (\xi_s - \hat{\xi}_s)^2 \nu_s ds \right] +
    \EE \left[
    \int_{[0,T)} c_s d[\hat{\xi}]_s \right], \label{eq:lowerbound1}
 \end{align}
 where we used monotone convergence and applied Doob's Optional
 Sampling Theorem as, e.g., in \citet{Prot:04}, Theorem I.16, in order
 to get $\EE[M_{\tau^n}(u)]=0$ for all $n \geq 1$. In particular,
 the computations in \eqref{eq:lowerbound1} show that
 \eqref{ass:integrability} necessarily holds true if $\cU^c \neq
 \varnothing$ (as assumed for now). In other words,
 setting
\begin{equation}
  v \set  c_0(x-\hat{\xi}_0)^2 + \EE\left[ \int_0^T
    (\xi_s - \hat{\xi}_s)^2 \nu_s ds \right] +
  \EE \left[
    \int_{[0,T)} c_s d[\hat{\xi}]_s \right] < \infty, \label{eq:vfinite}
\end{equation}
we have for all $u \in \cU^c$ the lower bound
\begin{equation}
  J(u) \geq v. \label{eq:lowerbound2}
\end{equation}
Now, let us define the control $\hat{u}$ with corresponding controlled
process $\hat{X} \set X^{\hat{u}}$ via the feedback law
\begin{equation*}
  \hat{u}_t = \frac{c_t}{\kappa_t} (\hat{\xi}_t - \hat{X}_t) \quad (0
  \leq t < T).
\end{equation*}
Observe that $\hat{u}$ is a progressively measurable process and
locally $dt$-integrable and locally $\kappa_tdt$-square-integrable on
$[0,T)$ due
to~\eqref{eq:condkappanu}. In particular,
$\hat{X}_T = x + \int_0^T \hat{u}_t dt$ exists $\PP$-a.s and we can
invoke Lemma~\ref{lem:mastereq}. We denote by
$C_t(\hat{u}) = c_0 (x-\hat{\xi}_0)^2 + M_t(\hat{u}) + A_t(\hat{u})$,
$0 \leq t < T$, the corresponding cost process from this lemma. We
will now show that $\hat{u} \in \cU^c$ and that $\hat{u}$ attains the
lower bound in \eqref{eq:lowerbound2}, i.e.,
\begin{equation*}
  J(\hat{u}) = v
\end{equation*}
finishing our verification argument. Indeed, first note that, by choice
of $\hat{u}$, we have
\begin{equation*}
  A_t(\hat{u})=\int_0^t (\xi_s - \hat{\xi}_s)^2 \nu_s ds + \int_0^t
  c_s d[\hat{\xi}]_s \quad (0 \leq t < T),
\end{equation*}
whence, in particular,
\begin{equation*}
  v = c_0(x-\hat{\xi}_0)^2 + \EE[A_{T-}(\hat{u})] < \infty.
\end{equation*}
Next, since $M(\hat{u})$ is a local martingale on $[0,T)$ by virtue of
Lemma \ref{lem:mastereq} above, we can fix a localizing sequence of
stopping times $(\hat{\tau}^n)_{n=1,2,\ldots}$ such that
$\hat{\tau}^n \uparrow T$ $\PP$-a.s. strictly from below for
$n \rightarrow \infty$ and such that
$(M_{t \wedge \hat{\tau}^n}(\hat{u}))_{0 \leq t < T}$ is a uniformly
integrable martingale for each $n$. Then, for any stopping time
$\tau < T$, applying Fatou's Lemma and once more
Doob's Optional Sampling Theorem yields
\begin{align*}
  \EE[C_{\tau}(\hat{u})] = 
  & ~\EE [ \liminf_{n \rightarrow \infty}
    C_{\tau \wedge \hat{\tau}^n}(\hat{u})] \leq \liminf_{n \rightarrow \infty}
    \EE[C_{\tau \wedge \hat{\tau}^n}(\hat{u})]  \\ 
  = & ~c_0(x-\hat{\xi}_0)^2 + \liminf_{n \rightarrow \infty} \left\{ 
      \EE[A_{\tau \wedge \hat{\tau}^n}(\hat{u})] +
      \EE[M_{\tau \wedge \hat{\tau}^n}(\hat{u})] \right\}  \\
  = & ~ c_0 (x-\hat{\xi}_0)^2 + \liminf_{n
      \rightarrow \infty} \EE[A_{\tau \wedge \hat{\tau}^n}(\hat{u})]
       \\
  = & ~c_0
      (x-\hat{\xi}_0)^2 + \EE[A_{\tau}(\hat{u})] \leq c_0
      (x-\hat{\xi}_0)^2 + \EE[A_{T-}(\hat{u})]= v, 
\end{align*} 
where we also used monotone convergence as well as the fact that
$(A(\hat{u})_t)_{0 \leq t < T}$ is an increasing process. Hence, it
holds that
\begin{equation}
  J(\hat{u}) = \limsup_{\tau \uparrow T} \EE[C_{\tau}(\hat{u})] \leq
  v < \infty \label{eq:upperbound}
\end{equation}
and thus $\hat{u} \in \cU^c$. In particular, due to
\eqref{eq:lowerbound2}, we actually have $J(\hat{u}) = v$ as
desired. 

Finally, let us assume that \eqref{ass:integrability} is
satisfied. Then, it follows from \eqref{eq:vfinite} and
\eqref{eq:upperbound} that $\hat{u} \in \cU^c$, i.e.,
$\cU^c \neq \varnothing$. In other words, condition
\eqref{ass:integrability} is not only necessary but also sufficient
for $\cU^c \neq \varnothing$. \qed \medskip

\begin{Remark} \label{rem:timeconsistent} 
  Let us briefly comment on a few insights offered by the preceding
  proof. The argument rests on the key identity~\eqref{eq:master} of
  Lemma~\ref{lem:mastereq}. For bounded coefficients
  $\kappa, 1/\kappa,\nu,\eta$ and bounded targets $\xi,\Xi_T$, the
  theory of BSRDEs readily allows one to deduce that $M(u)$ of
  \eqref{eq:M} is a true martingale for any control $u$ with finite
  expected costs; see, e.g., \citet{KohlmannTang:02}. From the key
  identity~\eqref{eq:master} it then transpires that the control
  $\hat{u}^c$ of~\eqref{eq:optimalcontrol} minimizes
   \begin{displaymath}
     \EE \left[ \int_0^\tau (X^u_t - \xi_t)^2 \nu_t dt + \int_0^\tau
       \kappa_t u^2_t dt + c_\tau (X^u_\tau - \hat{\xi}^c_\tau)^2
     \right]
  \end{displaymath}
  \emph{simultaneously} for all stopping times $\tau \leq T$. In that
  sense, the above terminal penalizations
  $c_\tau (X^u_\tau - \hat{\xi}^c_\tau)^2$ are consistent replacements
  for $\eta(X^u_T-\Xi_T)^2$ for these problems with shorter time
  horizons.

  When coefficients are unbounded, particularly when
  $\P[\eta=+ \infty]>0$, it is quite possible that $M(u)$ is a strict
  local martingale and so the preceding argument breaks down. Still,
  being bounded from below by an integrable random variable under
  condition~\eqref{ass:integrability}, $M(u)$ is a supermartingale;
  its martingale property thus turns out to hinge on the control in
  $L^1(\P)$ of $c_{\tau^n} (X^u_{\tau^n} - \hat{\xi}^c_{\tau^n})^2$
  along a suitable sequence of stopping times $\tau^n \uparrow
  T$. This control does not seem to be available in the BSRDE
  literature at present, leaving our conjecture~\eqref{eq:conjecture}
  still open at this point.

  Observe, though, that, taking the $\limsup_{\tau \uparrow T}$ of the
  above expectations, the formulation in our auxiliary target
  functional $J^c(u)$ of \eqref{eq:defobjective} avoids these issues
  and thus allows us to solve at least these closely related auxiliary
  problems.
\end{Remark}

The final lemma justifies the interpretation in Remark \ref{rem:osp}:

\begin{Lemma} \label{lem:remark} 
  Let us assume that $\lim_{t \uparrow T} L_t = L_T$ in $L^1(\PP)$ and
  that the local c\`adl\`ag martingale $(N_t)_{0 \leq t < T}$
  in~\eqref{eq:BSRDE} satisfies $\E[[N]^{1/2}_t] < \infty$ for all
  $0 \leq t < T$. Then we have the representation
  \begin{equation} \label{eq:repc1} 
    c_t = \EE \left[ L_T e^{\int_0^t\frac{c_u}{\kappa_u} du} + \int_t^T e^{-\int_t^r
        \frac{c_u}{\kappa_u}du} \nu_r dr \, \bigg\vert \, \cF_t
    \right] \quad (0 \leq t < T).
  \end{equation} 
  Moreover, on $\{ \PP[\int_t^T \nu_r dr = 0 \; \vert \; \cF_t] < 1\}$ we have the identity
    \begin{equation} \label{eq:repw1}
      \EE \left[ \int_t^T \frac{e^{-\int_t^r \frac{c_u}{\kappa_u}du}}{(1-w_t)c_t}
        \nu_r dr \bigg\vert \mathcal{F}_t \right] = 1,
    \end{equation}
    and the weight process $w_t = \EE[L_T \vert \cF_t] / L_t$
    of~\eqref{eq:weight} satisfies
    $0 \leq w_t < 1$.
\end{Lemma}

\begin{proof}
  Recall the dynamics of the process $(L_t)_{0 \leq t < T}$ in
  \eqref{eq:SDEL}, i.e.,
  \begin{equation*}
    L_t = c_0 - \int_0^t e^{-\int_0^r \frac{c_u}{\kappa_u}du} \nu_r dr
    - \int_0^t e^{-\int_0^r \frac{c_u}{\kappa_u}du}
    dN_r \quad (0 \leq t < T).
  \end{equation*}
  Hence, for all $0 \leq t \leq s < T$ we may write
  \begin{equation}
    L_s - L_t = -\int_t^s e^{-\int_0^r \frac{c_u}{\kappa_u}du} \nu_r dr
    - \int_t^s e^{-\int_0^r \frac{c_u}{\kappa_u}du} dN_r. \label{eq:repc2}
  \end{equation}
  Observe that the stochastic integral in \eqref{eq:repc2} is a
  martingale on $[0,T)$ by our integrability assumption on
  $[N]^{1/2}_t$. Thus, taking conditional expectations
  in~\eqref{eq:repc2} yields
  \begin{equation}
    \EE[ L_s \vert \cF_t] - L_t = - \EE \left[\int_t^s e^{-\int_0^r
        \frac{c_u}{\kappa_u}du} \nu_r dr \, \bigg\vert \, \cF_t
    \right] \quad (0 \leq t \leq s < T). \label{eq:repc3}
  \end{equation}
  Passing to the limit $s \uparrow T$ in \eqref{eq:repc3} we obtain,
  due to monotone convergence and due to the assumption that $L_s$
  converges in $L^1(\PP)$ to $L_T$, the representation
  \begin{equation}
    L_t = \EE \left[ L_T + \int_t^T e^{-\int_0^r
        \frac{c_u}{\kappa_u}du} \nu_r dr \, \bigg\vert \, \cF_t
    \right] \quad (0 \leq t < T). \label{eq:repc4}
  \end{equation}
  In other words, using that
  $L_t=c_t e^{-\int_0^t \frac{c_u}{\kappa_u} du}$, we can write
  \begin{equation*}
    c_t = \EE \left[ L_T e^{\int_0^t\frac{c_u}{\kappa_u} du} + \int_t^T e^{-\int_t^r
        \frac{c_u}{\kappa_u}du} \nu_r dr \, \bigg\vert \, \cF_t
    \right] \quad (0 \leq t < T)
  \end{equation*}
  as desired for \eqref{eq:repc1}. Finally, by definition of the
  weight process $(w_t)_{0 \leq t < T}$ in~\eqref{eq:weight} together
  with the identity in~\eqref{eq:repc1}, we can write
  \begin{align}
    w_t 
    & = \frac{\EE[L_T \vert \cF_t]}{L_t} 
      = \frac{ e^{\int_0^t \frac{c_u}{\kappa_u}
      du}}{c_t} \EE [ L_T \,\vert \, \mathcal{F}_t] =
      \frac{1}{c_t} \EE \left[ e^{\int_0^t \frac{c_u}{\kappa_u}
      du} L_T \,\Big\vert \, \mathcal{F}_t \right] \nonumber \\
    & = \frac{1}{c_t} 
      \left( c_t - \EE \left[ \int_t^T e^{-\int_t^r \frac{c_u}{\kappa_u} du}
      \nu_r dr \, \Big\vert \, \mathcal{F}_t \right] \right) \nonumber
    \\
    & = 1 - \frac{1}{c_t} \EE \left[ \int_t^T e^{-\int_t^r \frac{c_u}{\kappa_u} du}
      \nu_r dr \, \Big\vert \, \mathcal{F}_t \right] \quad
      \text{for all } 0 \leq t < T, \label{eq:repw2}
  \end{align}
  which yields our claim~\eqref{eq:repw1}. In particular,
  representation~\eqref{eq:repw2} also reveals that $0 \leq w_t < 1$
  $\PP$-a.s. on
  $\{ \PP[\int_t^T \nu_r dr = 0 \; \vert \; \cF_t] < 1\}$ for all
  $0 \leq t < T$.
\end{proof}

\section*{Appendix} 

In this appendix, we collect some results on the BSRDE
in~\eqref{eq:BSRDE} with terminal condition~\eqref{eq:BSRDEtc} which
may be of independent interest for the theory of BSDEs. First, let us
provide lower estimates for a minimal supersolution $c^{\min}$ to the
Riccati BSDE in~\eqref{eq:BSRDE} with terminal
condition~\eqref{eq:BSRDEtc}.


\begin{Lemma} \label{app:lem:lbound} 
  Let $(\nu_t)_{0 \leq t \leq T}$, $(\kappa_t)_{0 \leq t \leq T}$
  satisfy~\eqref{eq:condkappanu} and let $c^{\min}$ denote a minimal
  supersolution to~\eqref{eq:BSRDE} with terminal
  condition~\eqref{eq:BSRDEtc}. Then for all $t \in [0, T)$ we have
  \begin{equation} \label{app:lbound} 
    c^{\min}_t \geq \EE \left[
      \frac{1}{\int_t^T \frac{1}{\kappa_s} ds + \frac{1}{\eta}} \,
      \bigg\vert \, \cF_t \right] \geq 0 \quad \PP\text{-a.s.}
  \end{equation}
  with strict inequality holding true in the first estimate on
  $\{ \PP[\int_t^T \nu_s ds = 0\vert \cF_t] < 1 \}$ and strict
  inequality in the second estimate on
  $\{ \PP[\eta = 0 \vert \cF_t] < 1\}$. In particular, any
  supersolution $c$ of~\eqref{eq:BSRDE} and~\eqref{eq:BSRDEtc} will be
  strictly positive throughout $[0,T)$ if~\eqref{ass:etakappa} holds true.
\end{Lemma}

\begin{proof} 
  We will adopt the same idea as in the proof of Lemma 11
  in~\citet{Popier:06} in the case $\kappa \equiv 1$ (and
  $\nu \equiv 0$). For all $n \geq 1$ we define the
  processes 
  \begin{equation*}
    \Gamma^n_t \set \EE \left[ \frac{1}{\int_t^T \frac{1}{\kappa_s} ds +
        \frac{1}{\eta \wedge n}}
      \, \bigg\vert \, \cF_t \right] \quad (0 \leq t \leq T).
  \end{equation*}
  Note that $\Gamma^n$ is well defined because the term in the
  conditional expectation is bounded by $n$. Moreover, we have
  pathwise the identity
  \begin{equation*}
    \frac{1}{\int_t^T \frac{1}{\kappa_s} ds +
      \frac{1}{\eta \wedge n}} = \eta \wedge n - \int_t^T
    \frac{1}{\kappa_s} \left( \frac{1}{\int_s^T \frac{1}{\kappa_u} du +
        \frac{1}{\eta \wedge n}} \right)^2 ds \quad (0 \leq t \leq T).
  \end{equation*}
  Thus, the process $\Gamma^n$ verifies
  \begin{align*}
    \Gamma^n_t 
    & = \EE \left[ \eta \wedge n - \int_t^T
      \frac{1}{\kappa_s} \left( \frac{1}{\int_s^T \frac{1}{\kappa_u} du +
      \frac{1}{\eta \wedge n}} \right)^2 ds \, \Bigg\vert \, \cF_t
      \right] \\
    & = \EE \left[ \eta \wedge n - \int_t^T \frac{1}{\kappa_s} 
      \left( (\Gamma_s^n)^2 + U^n_s \right) ds \, \Big\vert \, \cF_t
      \right] \quad (0 \leq t \leq T)
  \end{align*}
  with adapted process $U^n$ given by
  \begin{equation*}
    U^n_s \set \EE \left[ \left( \frac{1}{ \int_s^T \frac{1}{\kappa_u} du +
          \frac{1}{\eta \wedge n} } \right)^2
      \, \Bigg\vert \, \cF_s \right] - (\Gamma^n_s)^2 \quad (0 \leq s
    \leq T).
  \end{equation*}
  Observe that
  \begin{equation*}
    d\Gamma^n_t = \left( \frac{(\Gamma_t^n)^2}{\kappa_t} +
      \frac{U^n_t}{\kappa_t} \right) + dM^n_t, \quad \Gamma^n_T = \eta
    \wedge n,
  \end{equation*}
  for some c\`adl\`ag local martingale $(M^n_t)_{0 \leq t \leq
    T}$. Moreover, since $U^n_t \geq 0$ for all $0 \leq t \leq T$ due
  to Jensen's inequality, we have
  \begin{equation*}
    -\frac{y^2}{\kappa_t} - \frac{U^n_t}{\kappa_t} \leq
    -\frac{y^2}{\kappa_t} \leq -\frac{y^2}{\kappa_t} + \nu_t \quad (y
    \in \mathbb{R}).
  \end{equation*}
  Thus, classical comparison results as in \citet{KrusePopier:16_2}, Proposition
  4, together with the construction of the minimal supersolution
  $(c^{\min}_t)_{0 \leq t < T}$ via a truncation procedure in
  \cite{KrusePopier:16_1}, finally yields that for all $t \in [0, T)$
  we have
  \begin{equation*}
    c^{\min}_t \geq \EE \left[ \frac{1}{\int_t^T \frac{1}{\kappa_s} ds +
        \frac{1}{\eta \wedge n}}
      \, \bigg\vert \, \cF_t \right] \geq 0\quad \PP\text{-a.s.}
  \end{equation*}
In fact on $\{ \PP[\int_t^T \nu_s ds = 0\vert \cF_t] < 1 \}$
comparison is strict in the first of these estimates. Moreover,
letting $n \rightarrow \infty$ we conclude~\eqref{app:lbound} where
``$>0$'' holds on $\{ \PP[\eta = 0 \vert \cF_t] < 1\}$.
\end{proof}

Finally, let us briefly discuss the integrability condition
in~\eqref{eq:BSRDEintcond2} for the minimal supersolution
$(c^{\min}_t)_{0 \leq t < T}$ with Riccati dynamics~\eqref{eq:BSRDE}
satisfying the terminal condition~\eqref{eq:BSRDEtc}. This condition
is not regularly discussed in the BSRDE literature and thus calls for
a verification in some sufficiently generic setting. So let us place
ourselves in the context of \citet{KrusePopier:16_1} and therein
restrict ourselves to a Brownian framework. It follows from
Proposition 3 and Remark 4 as well as Corollary~1 in
\cite{KrusePopier:16_1} with $p=2$ that for any $t \in [0,T)$ we have
the upper estimates
\begin{equation} \label{app:ubound}
  c^{\min}_t \leq \frac{1}{(T-t)^2} \EE \left[ \int_t^T (\kappa_s + (T-s)^2
  \nu_s) ds \, \bigg\vert \, \cF_t \right] \quad \PP\text{-a.s.}
\end{equation}
In addition to that, observe that also the lower estimates derived in
Lemma~\ref{app:lem:lbound} hold true.

For simplicity, let us further confine ourselves to the following
additional assumptions on $(\nu_t)_{0 \leq t \leq T}$,
$(\kappa_t)_{0 \leq t \leq T}$ and $\eta$: We assume that the process
$(\kappa_t)_{0 \leq t \leq T}$ is bounded from below and above, i.e.,
it holds that
\begin{equation} \label{app:condkappa} 
  0 < k \leq \kappa_t \leq K <
  \infty \quad (0 \leq t \leq T)
\end{equation}
for some constants $k,K \in \RR$.  Moreover, we assume that $\nu \in
L^1(\PP \otimes dt)$ with
\begin{equation} \label{app:condnu2}
\frac{1}{T-t} \EE \left[ \int_t^T (T-s)^2 \nu_s ds \, \bigg\vert \,
  \cF_t \right] \leq C \qquad (0 \leq t < T)
\end{equation}
for some constant $C < \infty$. Finally, we assume that there exists a
constant $\varepsilon > 0$ such that
\begin{equation} \label{app:condeta}
\PP \left[ \varepsilon \leq \eta \leq + \infty \right] = 1.
\end{equation}
Observe that condition~\eqref{app:condeta} implies in particular that
$c_t>0$ $\PP$-a.s. for all $t \in [0,T]$ by virtue of
Lemma~\ref{app:lem:lbound}.

\begin{Lemma} 
  Under the conditions \eqref{app:condkappa}, \eqref{app:condnu2}, and
  \eqref{app:condeta} the minimal supersolution
  $c \set (c^{\min}_t)_{0 \leq t < T}$ to the BSRDE
  in~\eqref{eq:BSRDE} on $[0,T)$ with terminal
  condition~\eqref{eq:BSRDEtc} satisfies
  \begin{equation*}
    \int_0^T \frac{d\langle c\rangle_t}{c_t^2} < \infty \quad \text{on
    the set } \{ \eta = + \infty \},
  \end{equation*}
  i.e., condition \eqref{eq:BSRDEintcond2} holds true.
\end{Lemma} 

\begin{proof}
  We extend the proof of Proposition 10 in \citet{Popier:06} done for
  the specific case $\kappa \equiv 1$ and $\nu \equiv 0$ to our more
  general setting by using the upper and lower bounds of the process
  $(c_t)_{0 \leq t < T}$ in \eqref{app:ubound} and
  \eqref{app:lbound}. First, note that conditions
  \eqref{app:condkappa} and \eqref{app:condeta} imply for the lower
  bound in \eqref{app:lbound} that
  \begin{equation} \label{app:proof:lp}
    c_t \geq \frac{k
      \varepsilon}{(T-t) \varepsilon + k} \quad (0 \leq t < T).
  \end{equation} 
 Concerning the upper bound in \eqref{app:ubound}, we obtain due to
 \eqref{app:condkappa} and \eqref{app:condnu2}
  \begin{equation} \label{app:proof:up}
    c_t \leq \frac{K + \text{const}}{T-t} \quad (0 \leq t < T).
  \end{equation}
  Since the process $c$ is bounded from below on $[0,T]$, we can apply
  It\^o's formula on $[0,T-\delta]$ for some $0 < \delta < T$ to the
  process $\sqrt{(T-t) c_t}$. Using the BSRDE dynamics of $c$ in
  \eqref{eq:BSRDE}, we obtain
  \begin{align*}
    0 & \leq \sqrt{(T-t) c_t} \\
      & = \sqrt{T c_0} + \int_0^t \left( \frac{\sqrt{T-s}}{2
        \sqrt{c_s}} \left(
        \frac{c_s^2}{\kappa_s} - \nu_s \right) 
        - \frac{\sqrt{c_s}}{2\sqrt{T-s}} \right) ds \nonumber \\
      & \hspace{12pt} 
        -\frac{1}{8} \int_0^t \frac{\sqrt{T-s}}{c_s^{3/2}} d\langle c
        \rangle_s - \frac{1}{2} \int_0^t \frac{\sqrt{T-s}}{\sqrt{c_s}} dN_s \\
      & = \sqrt{T c_0} + \frac{1}{2} \int_0^t \sqrt{T-s}
        \frac{\sqrt{c_s}}{\kappa_s} \left( c_s - \frac{\nu_s \kappa_s}{c_s}
        - \frac{\kappa_s}{T-s} \right) ds \nonumber \\
      & \hspace{12pt} 
        -\frac{1}{8} \int_0^t \frac{\sqrt{T-s}}{c_s^{3/2}} d\langle c
        \rangle_s - \frac{1}{2} \int_0^t \frac{\sqrt{T-s}}{\sqrt{c_s}} dN_s
        \qquad (0 \leq t \leq T-\delta)
  \end{align*} 
  and hence
  \begin{align}
    & \frac{1}{8} \int_0^{T-\delta} \frac{\sqrt{T-s}}{c_s^{3/2}} d\langle c
      \rangle_s + \frac{1}{2} \int_0^{T-\delta} \frac{\sqrt{T-s}}{\sqrt{c_s}} dN_s
      \nonumber \\
    & \leq \sqrt{T c_0} + \frac{1}{2} \int_0^{T-\delta} \sqrt{T-s}
      \frac{\sqrt{c_s}}{\kappa_s} \left( c_s - \frac{\nu_s \kappa_s}{c_s}
      - \frac{\kappa_s}{T-s} \right) ds \label{app:proof:ito}
  \end{align}
  for all $0 < \delta < T$. Observe that due to the bounds on $c$ in
  \eqref{app:proof:lp} and \eqref{app:proof:up} and $\kappa$ in
  \eqref{app:condkappa} as well as the integrability assumption on
  $\nu$, i.e., $\nu \in L^1(\PP \otimes dt)$, it holds for all $0 < \delta < T$ that
  \begin{align*}
    & \EE \left[ \int_0^{T-\delta} \sqrt{T-s}
      \frac{\sqrt{c_s}}{\kappa_s} \left\vert c_s - \frac{\nu_s \kappa_s}{c_s}
      - \frac{\kappa_s}{T-s} \right\vert ds \right] \\
    & \leq \text{const} \, \EE \left[ \int_0^{T-\delta} \left\vert c_s 
      - \frac{\nu_s \kappa_s}{c_s} - \frac{\kappa_s}{T-s} \right\vert ds
      \right] \\
    & \leq \text{const} \left( \, \EE \left[ \int_0^{T-\delta} c_s ds
      \right] + \EE \left[ \int_0^{T-\delta} \frac{\nu_s
      \kappa_s}{c_s} ds \right] + 
      \EE \left[ \int_0^{T-\delta} \frac{\kappa_s}{T-s} ds
      \right] \right) < \infty.
  \end{align*} 
  Hence, by using the upper bound on $c$ in \eqref{app:ubound} and
  Fubini's Theorem, we
  can compute
  \begin{align}
    & \EE \left[ \int_0^{T-\delta} \left( c_s - \frac{\nu_s \kappa_s}{c_s}
      - \frac{\kappa_s}{T-s} \right) ds \right] 
      \leq \EE \left[ \int_0^{T-\delta} \left( c_s 
      - \frac{\kappa_s}{T-s} \right) ds \right] \nonumber \\
    & \leq \EE \left[ \int_0^{T-\delta} \left( \frac{1}{(T-s)^2} 
      \EE \left[ \int_s^T (\kappa_u + (T-u)^2
      \nu_u) du \, \bigg\vert \, \cF_s \right] 
      - \frac{\kappa_s}{T-s} \right) ds \right] \nonumber \\
    & \leq \EE \left[ \int_0^{T-\delta} \frac{1}{(T-s)^2} \left( \int_s^T
      \kappa_u du \right) ds - \int_0^{T-\delta} \frac{\kappa_s}{T-s}
      ds \right] \nonumber \\
    & \hspace{15pt} +
      \EE \left[ \int_0^{T-\delta} \frac{1}{(T-s)^2} \left( \int_s^T
      (T-u)^2 \nu_u du \right) ds \right]. \label{app:proof:exp}
  \end{align} 
  Using once more Fubini's Theorem and the fact that $\kappa_t \leq K$
  for all $0 \leq t \leq T$, we get for the first expectation in
  \eqref{app:proof:exp} the estimate
  \begin{align}
    & \EE \left[ \int_0^{T-\delta} \frac{1}{(T-s)^2} \left( \int_s^T
      \kappa_u du \right) ds - \int_0^{T-\delta} \frac{\kappa_s}{T-s} ds
      \right] \nonumber \\
    & = \EE \left[ \int_0^{T-\delta} \frac{\kappa_u}{T-u} du +
      \int_{T-\delta}^T \frac{\kappa_u}{\delta} du - \frac{1}{T} \int_0^T
      \kappa_u du  - \int_0^{T-\delta} \frac{\kappa_s}{T-s} ds \right]
      \nonumber \\
    &  \leq K. \label{app:proof:exp1}
  \end{align} 
  Concerning the second expectation in \eqref{app:proof:exp},
  application of Fubini's Theorem yields
  \begin{align}
    & \EE \left[ \int_0^{T-\delta} \frac{1}{(T-s)^2} \left( \int_s^T
      (T-u)^2 \nu_u du \right) ds \right] \nonumber \\
    & \leq \EE \left[ \int_0^{T-\delta} (T-u) \nu_u du  + \delta \int_{T -
      \delta}^T \nu_u du \right]. \label{app:proof:exp2}
  \end{align} 
  Consequently, taking expectation in \eqref{app:proof:ito} and using
  that the stochastic integral with respect to $N$ in
  \eqref{app:proof:ito} is a true martingale on $[0,T-\delta]$ due to
  \eqref{app:proof:lp} and \eqref{app:condnu2}, we obtain
  together with the estimates in \eqref{app:proof:exp1} and
  \eqref{app:proof:exp2} the upper bound
  \begin{align*}
    & \frac{1}{8} \EE \left[ \int_0^{T-\delta} 
      \frac{\sqrt{T-s}}{c_s^{3/2}} d\langle c
      \rangle_s \right] \\
    & \leq \sqrt{T c_0} + \text{const} \,\left( K + 
      \EE \left[ \int_0^{T-\delta} (T-u) \nu_u du  + \delta \int_{T -
      \delta}^T \nu_u du \right] \right). 
  \end{align*}
  Passing to the limit $\delta \downarrow 0$ we get with monotone
  convergence
  \begin{align}
    & \EE \left[ \int_0^T \frac{\sqrt{T-s}}{c_s^{3/2}} d\langle c
      \rangle_s \right] \nonumber \\
    & \leq 8 \left( \sqrt{T c_0} + \text{const} \,\left( K +
      \EE \left[ \int_0^T (T-u) \nu_u du \right] \right) \right) <
      \infty, 
      \label{app:proof:intbound}
  \end{align}
  due to $\nu \in L^1(\PP \otimes dt)$. Now, using \eqref{app:lbound}, observe
  that we can further estimate the process $(c_t)_{0 \leq t < T}$
  from below by
  \begin{align*}
    c_s & \geq \EE \left[ \frac{1}{\int_s^T \frac{1}{\kappa_u} du 
          + \frac{1}{\eta}}
          \, \bigg\vert \, \cF_s \right] \geq 
          \EE \left[ \frac{1}{\int_s^T \frac{1}{\kappa_u} du +
          \frac{1}{\eta}} 1_{\{ \eta \, = \, + \infty \}}
          \, \bigg\vert \, \cF_s \right]  \\
        & = \EE \left[ \frac{1}{\int_s^T \frac{1}{\kappa_u} du } 1_{\{ \eta
          \, = \, + \infty \}}
          \, \bigg\vert \, \cF_s \right] \geq  \frac{k}{T-s} \EE \left[ 1_{\{
          \eta \, = \, + \infty \}}
          \, \vert \, \cF_s \right].
  \end{align*}
  Plugging back this lower bound into the left hand side of
  \eqref{app:proof:intbound} and using optional projection, we get
  \begin{align*}
    \infty & > \EE \left[ \int_0^T \frac{\sqrt{T-s}}{c_s^{3/2}} d\langle c
             \rangle_s \right] = \EE \left[ \int_0^T
             \frac{\sqrt{T-s}}{c_s^2} \sqrt{c_s} \, d\langle c
             \rangle_s \right] \\
           & \geq \sqrt{k} \, \EE \left[ \int_0^T
             \frac{1}{c_s^2} \EE \left[ 1_{\{
             \eta \, = \, + \infty \}}
             \, \vert \, \cF_s \right] \, d\langle c
             \rangle_s \right] = \sqrt{k} \, \EE \left[ \int_0^T
             \frac{1}{c_s^2} 1_{\{\eta \, = \, + \infty \}} \, d\langle c
             \rangle_s \right] \\
           & = \sqrt{k} \, \EE \left[ 1_{\{\eta \, = \, +
             \infty \}} \left(\int_0^T
             \frac{1}{c_s^2} \, d\langle c
             \rangle_s \right) \right],
  \end{align*}
  which yields the desired result.
\end{proof}


\bibliographystyle{plainnat} \bibliography{finance}

\end{document}